\documentclass[12pt]{article}

\usepackage{natbib}
\usepackage{amsthm}
\usepackage{amssymb}
\usepackage{amsmath}
\usepackage{amsfonts}
\usepackage{times}
\usepackage{enumitem}
\usepackage[colorlinks]{hyperref}
\usepackage{pdfsync}

\usepackage[paper=a4paper, top=2.5cm, bottom=2.5cm, left=2.5cm, right=2.5cm]{geometry}

\def\qed{\hfill $\vcenter{\hrule height .3mm
\hbox {\vrule width .3mm height 2.1mm \kern 2mm \vrule width .3mm
height 2.1mm} \hrule height .3mm}$ \bigskip}

\def \I{\mathcal{I}}
\def \prodm{\xi}
\def \tilt{ \mathrm{\tau} } 
\def \W{\mathrm{W_1}}
\def \WW{\mathrm{W_2}}
\def \Sph{\mathbb{S}^{n-1}}
\def \RR {\mathbb R}

\def \EE {\mathbb E}

\def \CC {\mathbb C}
\def \PP {\mathbb P}
\def \eps {\varepsilon}

\def \Lip {\mathrm {Lip}}

\def \TTT {\mathcal{T}}

\def \FF {\mathcal{F}}
\def \FFF {\mathcal{\tilde F}}

\def \tr {\mathrm{Tr}}

\def \DC {\mathcal{C}_n}
\def \CC {\overline{\mathcal{C}}_n}
\def \MW {\mathbf{GW}}
\def \Ent {\mathrm{Ent}}
\def \KL {\mathrm{D}_{\mathrm{KL}}}
\def \HH {\mathcal{H}}
\def \Comp { \mathcal{D} }

\newtheorem{theorem}{Theorem}
\newtheorem{lemma}[theorem]{Lemma}

\newtheorem{fact}[theorem]{Fact}
\newtheorem{observation}[theorem]{Observation}

\newtheorem{proposition}[theorem]{Proposition}
\newtheorem{corollary}[theorem]{Corollary}
\theoremstyle{definition}
\newtheorem{definition}[theorem]{Definition}
\theoremstyle{remark}
\newtheorem{remark}{Remark}

\long\def\symbolfootnotetext[#1]#2{\begingroup
\def\thefootnote{\fnsymbol{footnote}}\footnotetext[#1]{#2}\endgroup}

\begin{document}

\title{Gaussian-width gradient complexity, reverse log-Sobolev inequalities and nonlinear large deviations}
\date{}
\author{Ronen Eldan\thanks{Weizmann Institute of Science. Incumbent of the Elaine Blond Career Development Chair. Partially supported by the Israel Science Foundation (grant No. 715/16).}}
\maketitle

\begin{abstract}
We prove structure theorems for measures on the discrete cube and on Gaussian space, which provide sufficient conditions for mean-field behavior. These conditions rely on a new notion of complexity for such measures, namely the Gaussian-width of the gradient of the log-density. On the cube $\{-1,1\}^n$, we show that a measure $\nu$ which exhibits low complexity can be written as a mixture of measures $\{\nu_\theta\}_{\theta \in \mathcal{I}}$ such that: \textbf{i.} for each $\theta$, the measure $\nu_\theta$ is a small perturbation of $\nu$ such that $\log \tfrac{d \nu_\theta}{d \nu}$ is a linear function whose gradient is small and, \textbf{ii.} $\nu_\theta$ is close to some product measure, in Wasserstein distance, for most $\theta$. Thus, our framework can be used to study the behavior of low-complexity measures beyond approximation of the partition function, showing that those measures are roughly mixtures of product measures whose entropy is close to that of the original measure. In particular, as a corollary of our theorems, we derive a bound for the na\"ive mean-field approximation of the log-partition function which improves the nonlinear large deviation framework of \cite{CD14} in several ways: 1. It does not require any bounds on second derivatives. 2. The covering number is replaced by the weaker notion of Gaussian-width 3. We obtain stronger asymptotics with respect to the dimension. Two other corollaries are decomposition theorems for exponential random graphs and large-degree Ising models. In the Gaussian case, we show that measures of low-complexity exhibit an almost-tight reverse Log-Sobolev inequality. 
\end{abstract}

\tableofcontents

\section{Introduction}
Let $\mu$ be a measure on the discrete hypercube $\DC= \{-1,1\}^n$. In this work, we are interested in the following quesiton: Under what natural conditions does this measure admit an approximate decomposition into a mixture of product measures most of which having roughly the same entropy as the measure $\mu$? This form of simplicity is a strong manifestation of what is referred to in the statistical mechanics literature as \textit{mean-field} behavior.

Our main theorem provides a sufficient condition for such behavior, using a new notion of complexity, namely \textit{Gaussian-width gradient complexity}. We say that a measure has low-complexity if one has nontrivial bounds on the Gaussian width of the gradient of its log-density (this is made rigorous and quantitative below). Our definition is inspired by \cite{CD14}, where covering numbers are considered. 

Our main theorem (Theorem \ref{thm:maindecomp} below) shows that for a measure $\mu$ on $\DC$ with log-Lipschitz density, a low-complexity condition implies the existence of an approximate decomposition into product measures as described above. Additionally, these measures can be written as small \emph{tilts} of the original measure, namely, they can be attained by applying a change of density with respect to some log-linear function whose gradient is small.

Perhaps the most studied manifestation of mean field behavior is the approximation of the partition function, up to first order, via a product measure. More precisely, defining $\DC= \{-1,1\}^n$ equipped with the uniform measure $\mu$, the Gibbs variational principle states that
\begin{equation}\label{eq:Gibbs}
\log \int e^f d \mu = \sup_{\nu} \left ( \int f d \nu - \KL(\nu \Vert \mu) \right )
\end{equation}
where the supremum is taken over all probability measures $\nu$ on $\DC$ and $\KL$ denotes the Kullback-Leibler divergence (defined below). The \emph{na\"ive mean-field} approximation is said to hold true when the supremum is approximately saturated by the class of product measures. 

Suppose that the quantity $\log \int e^f d \mu$ is of order $O(n)$. One is often interested in cases where the approximation holds in first order, hence, the above inequality is saturated by product measures up to an error of $o(n)$. This sort of approximation corresponds to the case that the function $f$ is correlated with a linear function in a region whose measure is at least $\exp(-o(n))$. The main theorem of \cite{CD14} gives a sufficient condition for such an approximation to hold true. Our work takes another step, giving sufficient conditions for $f$ to be correlated with a (relatively small) family of linear functions \textit{almost-everywhere}. In physical terms, whereas the approximation for the partition function is equivalent to the existence of a single pure-state of non-negligible probability, our result gives a decomposition of the entire measure into pure states. As shown in an example below, replacing the unique product measure by a family thereof is necessary.

Our structure theorem applies to several settings, including subgraph-counting functions in random graphs, density of arithmetic progressions, mean-field Ising and Potts models and exponential random graphs (see below for background and references). A central corollary of the above-mentioned estimate for the partition function is a general framework deriving large deviation principles for nonlinear functions of Bernoulli random variables (Theorem \ref{thm:largedev} below) which extends the one in \cite{CD14} and improves the bounds in the examples considered there. Our framework provides a seemingly cleaner theorem which, in particular, does not require any assumptions on second derivatives.

A central example where the large deviations framework comes in handy is in the derivation of a large deviation principle for the number of triangles (or more generally, subgraph densities) in an Erd\"os-R\'enyi random graph $G=G(N,p)$. Letting $T$ denote the number of triangles in $G$, the goal is to find precise asymptotics for $\log \PP(T \geq (1+\delta) \EE T)$ as $N \to \infty$ (with $p$ possibly depending on $N$). For background and history concerning large deviations for random graphs, we refer the reader to the book \cite{Chattejee-LDRG} and references therein. When the function $f$ in \eqref{eq:Gibbs} is the number of triangles $T$, it turns out that when maximizing over product measures $\nu$, the right hand side becomes a tractable quantity which can be calculated almost precisely, as was done in \cite{LZ14} for the case of triangles, and later in \cite{BGLZ-Uppertails} for general subgraph counts. As we will demonstrate, applying our framework to the these examples seems to be a rather simple task that requires significantly less technical work compared to previous works.

In a subsequent work, \cite{EldanGross-Decomp}, our methods are used to derive a theorem showing that these product measures are close to critical points of the associated mean-field functional, giving a more precise characterization of the mixture.

\subsection{Main structure theorems}
To formulate our results, let us start with some definitions. Consider the discrete cube $\DC = \{-1,1\}^n$ equipped with the uniform probability measure $\mu$. First, we would like to define a notion of \emph{complexity} of a function $f:\DC \to \RR$. To this end, we first define the \emph{Gaussian-width} of a set $K \subset \RR^n$ as 
$$
\MW(K) = \EE \left [ \sup_{x \in K} \langle x, \Gamma \rangle \right ]
$$
where $\Gamma \sim N(0, \mathrm{Id})$ is a standard Gaussian random vector in $\RR^n$. Next, for a function $f:\DC \to \RR$ and a point $y = (y_1,...,y_n) \in \DC$ and $i \in [n]$, we write
$$
\partial_i f(y) = \frac{1}{2} \left (f(y_1,\dots, y_{i-1}, 1, y_{i+1}, \dots y_n) - f(y_1,\dots, y_{i-1}, -1, y_{i+1}, \dots y_n) \right )
$$
and define the discrete gradient of $f$ as
\begin{equation}\label{def:discg}
\nabla f(y) = (\partial_1 f(y), \dots, \partial_n f(y)).
\end{equation}
we will also define
$$
\Lip(f) = \max_{i \in [n], y \in \DC} | \partial_i f(y)|,
$$
the discrete Lipschitz constant of $f$.

Finally, for a function $f:\DC \to \RR$, the \emph{gradient-complexity} of $f$ will be defined as
\begin{equation}\label{eq:defcomplexity}
\Comp(f) := \MW \left (  \left \{\nabla f(y): ~ y \in \DC  \right \} \cup \{0\} \right )
\end{equation}
and for a measure $\nu$ on $\DC$, by slight abuse of notation, we define its complexity as 
$$\Comp(\nu) := \Comp \left (\log \frac{d \nu}{d \mu} \right ).
$$

\begin{remark}
In the following, the main regime which is of interest to us is functions $f$ which takes values of order $O(n)$ and whose Lipschitz constant is $\Lip(f) = O(1)$. It is clear that such functions trivially have complexity at most $O(n)$. The functions for which our results will be nontrivial are the ones whose complexity is $o(n)$.
\end{remark}
	
For two measures $\nu_1,\nu_2$ on $\DC$ we define the Wasserstein mass-transportation distance between $\nu_1,\nu_2$ as
$$
\W(\nu_1,\nu_2) = \inf_{(X,Y) \mbox{ s.t. } \atop X \sim \nu_1, Y \sim \nu_2} \EE d_H(X,Y) 
$$
where $d_H$ denotes the Hamming distance. We say that a measure $\xi$ on $\DC$ is a \emph{product measure} if $X_1,\dots,X_n$ are independent where $(X_1,\dots,X_n) \sim \xi$.

Roughly speaking, our first result states that if $\nu$ is a measure of low complexity then there exists another measure $\tilde \nu$ whose log-density is close to that of $\nu$ in $L_\infty$ and such that its $\W$-distance to some product measure $\xi$ is small. 
\begin{theorem} \label{thm:main}
Let $\nu$ be a probability measure on the discrete cube $\DC$. Then for every {\small $\eps \in \left (0, 1/16 \right )$}, there exists a probability measure
$\tilde \nu$ such that if we define the functions $f,\tilde f$ by the equations $\frac{d \nu}{d \mu} = e^{f}$ and $\frac{d \tilde \nu}{d \mu} = e^{\tilde f}$ then we have
\begin{equation}\label{eq:ftildef}
|f(y) - \tilde f(y)| \leq \eps n, ~~ \forall y \in \DC
\end{equation}
and for the unique \emph{product} measure $\xi$ on $\DC$ satisfying $\int y d\xi(y) = \int y d \tilde \nu(y)$, one has that
$$
\W(\tilde \nu, \xi) \leq 2^7 \sqrt{\frac{n \Comp(\nu)}{\eps}} + 4 n e^{-\frac{1}{64 \eps^2}}.
$$
Furthermore, there exists a set $I \subseteq [n]$ with $|I| \geq n - 2 n e^{-\frac{1}{32 \eps^2}}$ such that the following holds. For a measure $\rho$ on $\DC$, denote by $\pi_I (\rho)$ the marginal law of $\rho$ on the subset $I$. Let $\xi' := \pi_I(\xi) \times \pi_{[n] \setminus I}(\tilde \nu)$. Then one has
$$
\W(\tilde \nu, \xi') \leq 2^6 \sqrt{\frac{n \Comp(\nu)}{\eps}}.
$$
\end{theorem}

\begin{remark} 
It is important to emphasize that the function $\tilde f$ provided by the above theorem is by no means unique. For the sake of intuition, a good example to consider is when $f(y) = - \beta T(G)$, where $T(G)$ is the number of triangles in the $N$-vertex graph whose edge set is defined by the point $y \in \DC$, $n = {N \choose 2}$ and $\beta = 1/N$. In this case, one expects the measure $\nu$ to be (in a rough sense) close to a distribution on approximately-bipartite graphs. It is clear by symmetry that this distribution is invariant under permutations of the vertices. The choice of the function $\tilde f$ in the above theorem should then correspond to the choice of the partition, under which the edges should be approximately independent. Note that the entropy associated the choice of this partition is of the order $N$, which is significantly smaller than the entropy we expect to have left after that choice, which is of order $n$.
\end{remark}

Next, we would like to formulate an easy corollary to the above, which will be useful in the context of large deviation theory. For two probability measures $\nu_1,\nu_2$ on $\DC$, we define the Kullback-Leibler divergence of $\nu_1$ with respect to $\nu_2$ as
$$
\KL(\nu_1 \Vert \nu_2) = \int \log \left (\tfrac{d \nu_1}{d \nu_2} (y) \right ) d \nu_1(y).
$$
The corollary is analogous to \cite[Theorem 1.6]{CD14}. It reads:
\begin{corollary} \label{cor:main}
For every $f:\DC \to \RR$, there exists a \emph{product} probability measure $\xi$ on $\DC$ which satisfies
$$
\log \int_{\DC} e^{f} d \mu \leq \int f d \xi - \KL(\xi \Vert \mu) +  64 \Lip(f)^{2/3} \Comp(f)^{1/3} n^{2/3}.
$$
\end{corollary}
\begin{remark} \label{rem:continuity}
One can strengthen the above Corollary in the sense that $\Lip(f)$ can be replaced by a weaker notion of continuity. An inspection of the proof reveals that for any monotone, bounded and continuous function $\varphi:[0,1] \to [0,\infty)$ with $\varphi(0) = 0$, the following is true: Under the assumption $\Comp(f) = o(n)$, one gets a nontrivial mean-field approximation and as long as $f$ attains property that for all $x,y \in \DC$, one has $\frac{|f(x)-f(y)|}{n} \leq \varphi \left ( \frac{ d_H(x,y)}{n} \right )$. In other words, $f$ does not have to be Lipschitz in a local sense, it only needs to have small oscillations on some mesoscopic scale.
\end{remark}

We now turn to the formulation of our main structure theorem. For a measure $\nu$ on $\DC$ such that $d \nu = e^{f} d \mu$ and for a point $\theta \in \RR^n$ we define the \emph{tilt} of $\nu$ with respect to $\theta$, denoted by $\tilt_\theta \nu$, by the equation
$$
\frac{d (\tilt_\theta \nu)}{d \mu} (y) = \frac{ e^{f(y) + \langle \theta, y \rangle} }{ \int_{\DC} e^{f(z) + \langle \theta, z \rangle} d \mu }.
$$
Moreover, for every measure $\nu$ on $\DC$, define by $\prodm(\nu)$ to be the unique product measure having the same marginals as $\nu$. Define also $B(x_0,r) := \{x \in \RR^n; \|x-x_0\|_2 \leq r \}$.

Our main structure theorem states that any measure $\nu$ of low complexity admits a decomposition into small tilts which are close to product measures. Our theorem reads,
\begin{theorem} \label{thm:maindecomp}
	Let $\nu$ be a probability measure on the discrete cube $\DC$. For every {\small $\eps \in \\$ $\left (0, \tfrac 1 8 \log \left ( \tfrac{4n}{\Comp(\nu)} \right )^{-1/2} \right  )$} and $\alpha > 1$, there exists a measure $m$ supported on $B(0, \eps \sqrt{n}) \cap [- 1 ,1 ]^n $ such that $\nu$ admits the decomposition
	\begin{equation}\label{eq:decomp1}
	\int_{\DC} \varphi d \nu = \int_{B(0, \eps \sqrt{n})} \left (\int_{\DC} \varphi d (\tilt_\theta \nu) \right )  d m(\theta)
	\end{equation}
	for every test function $\varphi: \DC \to \RR$ and such that there exists $\Theta \subset \RR^n$ with $m(\Theta) > 1-\frac{1}{n}-\frac{1}{\alpha}$ so that for every $\theta \in \Theta$ one has
	\begin{equation}\label{eq:thmmd-2}
	\W \bigl (\tilt_\theta \nu, \xi_\theta \bigr ) \leq 16 \sqrt{\frac{\alpha n \Comp(\nu) }{\eps}},
	\end{equation}
	for some product measure $\xi_\theta$. Moreover, we have that
	\begin{equation}\label{eq:thmmd-3}
	\left | \KL(\nu \Vert \mu) - \int \KL( \tilt_\theta \nu \Vert \mu ) d m(\theta) \right | \leq 2 \eps n.
	\end{equation}
\end{theorem}
\begin{remark}
	In the subsequent work \cite{EldanGross-Decomp} it is shown that, under an extra technical condition involving the second derivatives of $\log \frac{d \nu}{d \mu}$, the measures $\xi_\theta$ are critical points of the associated Gibbs functional. 
\end{remark}
We move on to formulating our theorem for the Gaussian case. Denote by $\gamma$ the standard Gaussian measure on $\RR^n$. For a differentiable function $f:\RR^n \to \RR$ we define the complexity of $f$ as $\Comp(f) = \MW(\{\nabla f(x): x \in \RR^n \})$. Let $\nu$ be a density with respect to $\gamma$ such that $d \nu = e^f d \gamma$. In this case we define
$$
\KL(\nu \Vert \gamma) = \int f d \nu, ~~ I(\nu) = \int |\nabla f|^2 d \nu
$$
the Kullback-Leibler divergence and the Fisher information of of $\nu$. The log-Sobolev inequality on Gaussian space asserts that for every measure $\nu$,
$$
\I(\nu) \geq 2 \KL(\nu \Vert \gamma). 
$$
The following theorem reverses this inequality for measures of low complexity.

\begin{theorem}\label{thm:mainGaussian}
Let $\nu$ be a measure on $\RR^n$, such that $f = \log \frac{d \nu}{d \gamma}$ for some twice-differentiable function $f$. One has
$$
\I(\nu) - 2 \KL(\nu \Vert \gamma) \leq 2 \Comp(\nu)^{2/3} \I(\nu)^{1/3} + \max \left (- \inf_{x \in \RR^n} \Delta f(x), 0 \right).
$$
\end{theorem}

\subsection{A large deviation framework for functions of low complexity}
We now turn to formulating our main theorem concerning nonlinear large deviations, which is parallel to \cite[Theorem 1.1]{CD14}. Fix a function $f: \DC \to \RR$. For $0\leq p \leq 1$, define $\mu_p$ to be the measure whose density is
$$
\frac{d \mu_p}{d \mu} (y_1,\dots,y_n) = \prod_{i \in [n]} \left (1 - y_i (1- 2 p) \right ).
$$
Our central definition is the rate function
\begin{equation}\label{eq:defphi}
\phi_p(t) = \inf_{\xi \in \mathcal{PM}(\DC)} \left \{  \KL( \xi \Vert \mu_p) : ~ \int f d \xi \geq t n  \right \}.
\end{equation}
where $\mathcal{PM}(\DC)$ is the space of product probability measures over $\DC$.
\begin{theorem}\label{thm:largedev}
Let $p \in (0,1)$ and let $Y \sim \mu_p$. Then for every $t,\delta \in \RR$ which satisfy $0 < \delta < \tfrac{1}{n} \phi_p(t-\delta)$, we have the bound
$$
\log \PP(f(Y) \geq tn) \leq - \phi_p(t-\delta) \left (1 - 64 L n^{-1/3} \right )
$$
with
{ \small
$$
L = \frac{1}{\delta} \bigl ( 2 \Lip(f) + |\log(p(1-p))| \bigr )^{2/3} \left ( 2 \Comp(f) + \frac{1}{\delta} \Lip(f)^2 \right )^{1/3}.
$$ }
Moreover, whenever the assumption $\frac{1}{n \delta^2 } \Lip(f)^2 \leq \tfrac{1}{2}$ holds, we also have the lower bound
$$
\log \PP(f(Y) \geq (t - \delta) n) \geq -\phi_p(t) \left ( 1 + \frac{2}{n \delta^2 } \Lip(f)^2 \right ) - 2.
$$
\end{theorem}

\begin{remark}
If the function $f$ is $O(1)$-Lipschitz, the above theorem shows that one is able to obtain a nontrivial bound with some $\delta \to 0$ as long as $\Comp(f)/n$ tends to $0$.
\end{remark}

\subsubsection{An example application: triangles in $G(N,p)$} \label{sec:triangles}
To illustrate how the above theorem can be applied, let us use it to derive a large deviation principle for the number of triangles in $G(N,p)$. A second example application of the framework is to large deviations of the number of arithmetic sequences for random subsets of $\mathbb{Z} / n \mathbb{Z}$, which we will not discuss here, was carried out in \cite{BBGSZ16}. 

Let $\mathcal{P}$ denote upper triangular arrays of the form $y = (y_{i,j})_{1 \leq i<j \leq N}$ where $y_{i,j} \in \{-1,1\}$. We associate every $x \in \mathcal{P}$ with the undirected graph $G_y=([N], E)$ where, for $i < j$ we have $(i,j) \in E$ if and only if $y_{i,j} = 1$. We will also understand $y$ as a point in $\DC$ with $n= {N \choose 2}$.

Define $f(y) = \frac{1}{N} T(G_y)$ where $T(G)$ is the number of triangles in $G$. Define also $A(y)$ to be the adjacency matrix of $G_y$, which is in other words the unique symmetric matrix whose above-diagonal half determined by $y$. Moreover for a symmetric matrix $A$ define $u(A) \in \RR^{N \choose 2}$ to be the above-diagonal vector associated with $A$, so that $u(A)_{i,j} = A_{i,j}$ for $i<j$. It is easily checked that $f(y) = \frac{1}{6 N} \tr(A(y)^3)$ and $\nabla f(y) = \frac{1}{N} u(A(y)^2)$. Define
$$
\mathcal{A} = \left \{ \frac{B^2}{N}: B \in \mathbb{M}_{N \times N}, B=B^T, |B_{i,j}| \leq 1, \forall i,j \in [N]  \right \}.
$$
Clearly $\{ \nabla f(y): y \in \DC \} \in u(\mathcal{A})$. We would like to bound $\MW(u(\mathcal{A}))$. First remark that for all $A \in \mathcal{A}$ we have that $A$ is positive definite with $\tr(A) \leq N$, so the Schatten $1$-norm of $A$ is bounded by $N$. The noncommutative H\"older inequality therefore gives that for every $N \times N$ matrix $Q$, one has
$$
\sup_{A \in \mathcal{A}} \tr(AQ) \leq N \Vert Q \Vert_{\mathrm{OP}}.
$$
Now let $\Gamma = (\Gamma_{i,j})_{1 \leq i < j \leq N}$ be a standard Gaussian random vector in $\RR^{{N \choose 2}}$ and define by $M(\Gamma)$ the unique symmetric $N \times N$ matrix whose diagonal is zero and whose above-diagonal part is equal to $\Gamma$. Then we have by the above inequality,
$$
\EE \left [\sup_{y \in \DC} \langle \nabla f(y), \Gamma \rangle \right ] \leq \EE \left [\sup_{A \in \mathcal{A}} \tr(A M(\Gamma)) \right ] \leq N \EE \Vert M(\Gamma) \Vert_{\mathrm{OP}} \leq 2 N^{3/2}
$$
where the last inequality is well-known, and follows for example from an application of Slepian's lemma. Moreover, remark that the entries of $B^2/N$ above are bounded by $1$. We conclude the following,
\begin{fact}
One has that $\Comp(f) \leq 5 n^{3/4}$ and $\Lip(f) \leq 1$.
\end{fact}

This proof is easily generalized for any subgraph count, see Lemma \ref{lem:subgraphcomplexity} below. An application of Theorem \ref{thm:largedev} gives for all $t$ and $n^{-1/2} < \delta < \tfrac{1}{n} \phi_p(t-\delta)$,
$$
- \phi_p(t+\delta) \left ( 1 + \frac{2}{n \delta^2} \right ) \leq \log \PP(f(Y) \geq tn) \leq - \phi_p(t-\delta) \left (1 - 64 L n^{-1/3} \right )
$$
with
$$
L = \frac{20}{\delta} \left ( 3 + |\log(p(1-p))| \right )^{2/3}  n^{1/4}.
$$ 
We now state the solution to the variational problem obtained by Lubetzky and Zhao,
\begin{theorem}[\cite{LZ14}]
If $N^{-1/2} \ll p_N \ll 1$ then one has
$$
\lim_{N \to \infty} \frac{\phi_{p_N}((\alpha+1) p_N^3 )}{{N \choose 2} p_N^2 \log(1/p_N)} = \min \left (\alpha^{2/3}, \frac{2}{3} \alpha \right ).
$$
\end{theorem}
Let $p_N \in (0,1)$ be some sequence depending on $N$. Fix $\alpha > 0$ and define $t_N = (1 + \alpha) p_N^3$ and $\delta_N = p_N^3 / (\log \log N)$. Thus, the assumption 
$$
p_n \gg N^{-1/18} \log(N)
$$ 
finally gives 
$$
\log \PP \left (T(G(N,p_N)) \geq (1+\alpha) p_N^3 { N \choose 3 } \right ) = \log \PP \left (f(Y) \geq (1+\alpha) p_N^3 \frac{n}{6} \right ) = - \phi_p(t_N) (1 + o(1)).
$$

\subsection{Mean-field behavior of the Ising model with large degree}

In this section we demonstrate how our framework can be used to study the behavior of the Ising model satisfying a mean-field assumption in the spirit of \cite{BM16}. For the sake of simplicity, we will only discuss the Ising model, however our methods work for the Potts model as well.

Let $V=[n]$ be a set of sites, and consider a spin system taking configurations $\sigma \in \{-1,1\}^V$. Let $A= \left (A_{i,j} \right )_{i,j=1}^n$ be a real-valued, symmetric interaction matrix, and $b = (b_i)_{i=1}^n$ a vector of magnetic moments $b_i \in \RR$. Consider the Hamiltonian
$$
f(\sigma) = \langle \sigma, A \sigma \rangle + \langle b, \sigma \rangle.
$$
Define $\nu$ to be the probability measure whose density is $\frac{d \nu}{d \mu} = f - \log Z$ with $Z$ being the normalizing constant. In order to use our framework, let us try to calculate the complexity of $f$. To this end, fix $\sigma \in \DC$ and $i \in [n]$. Write $\sigma_\pm = (\sigma_1,\dots,\sigma_{i-1},\pm 1,\sigma_{i+1},\dots,\sigma_n)$. We have
$$
2\partial_i f(\sigma) = \langle \sigma_+, A \sigma_+ \rangle -  \langle \sigma_-, A \sigma_- \rangle + \langle b, \sigma_+ - \sigma_- \rangle = 2 \sum_{j \in [n] \setminus \{i\}} A_{i,j} \sigma_j + 2 b_i.
$$
With the legitimate assumption $A_{i,i} = 0$ for $i \in [n]$, we get
$$
\nabla f(\sigma) = A \sigma + b.
$$
We therefore have
\begin{align*}
\Comp(f) ~& = \EE \sup_{\sigma \in \DC} \langle A \sigma + b, \Gamma \rangle \\
~& \leq \EE \sup_{\sigma \in \DC} |\langle A \sigma, \Gamma \rangle| + \EE |\langle b, \Gamma \rangle | \\
~& \leq \sqrt{n} \EE \sup_{x \in B(0,1)} |\langle A x, \Gamma \rangle| + \|b \|_2 \\
~& = \sqrt{n} \EE \|A \Gamma\|_2 + \|b \|_2  \\
~& \leq \sqrt{n \EE \| A \Gamma \|_2^2} + \|b \|_2 = \sqrt{n \left (\tr A^2 + b_{max}^2 \right ) } 
\end{align*}
where $b_{max} := \max_{i \in [n]} |b_i|$. The assumption $\tr(A^2) = o(n)$ is referred to in \cite{BM16} as the \emph{mean-field} assumption. \\
Next, we also have that $\Lip(f) \leq U(A) + b_{max}$ where
$$
U(A) = \max_{i \in [n], \sigma \in \DC} |\langle A \sigma, e_i \rangle| = \max_{i \in [n]} \sum_{j \in [n]} |A_{i,j}|.
$$

Invoking Corollary \ref{cor:main}, we get the following mean-field approximation for $Z$: there exists a product probability measure $\xi$ on $\DC$ which satisfies
$$
\frac{1}{n} \left |\log Z - \left (\int f d \xi - \KL(\xi \Vert \mu) \right ) \right | \leq 64 \left ( \frac{ \left ( U(A) + b_{max} \right )^{4} \left (\tr A^2 + b_{max}^2 \right )}{n} \right )^{1/6}.
$$
In particular, if $U(A) = O(1)$, $b_{max} = O(1)$ and $\tr A^2 = o(n)$, then
$$
\frac{1}{n} \left (\log Z - \int f d \xi + \KL(\xi \Vert \mu) \right ) = o(1).
$$
The result in \cite{BM16} relies on a weaker condition than $U(A) = O(1)$, namely that $\sup_{\sigma \in \DC} \left \| A \sigma \right \|_1 = O(n)$. However, following remark \ref{rem:continuity}, it is not hard to check that our assumption $U(A) = O(1)$ can be replaced by a weaker assumption: it is enough to assume, for instance, that there exists some $p>1$ such that  $\sup_{\sigma \in \DC} \left \| A \sigma \right \|_p = O(n^{1/p})$.

Aside from the approximation of the partition function, our framework gives more information about the behavior of $\nu$. Under the above conditions, an application of Theorem \ref{thm:maindecomp} tells us that $\nu$ can be approximately decomposed to a mixture of product measures, whose typical entropy is very close to the entropy of the system. 

\subsection{A decomposition theorem for exponential random graphs} \label{sec:expgraphs}

The goal of this subsection is to demonstrate an application of Theorem \ref{thm:maindecomp} to exponential random graphs. Loosely speaking, the theorem below states that an exponential random graph is close in Hamming distance to a random graph which can be expressed as a mixture of graphs with independent edges, in a way that most of the entropy comes from the independent graphs (rather than from the mixture).

For a two finite graphs $H,G$ with $m,N$ vertices respectively, we denote by $\mathrm{Hom(H,G)}$ the number of homomorphisms from the vertex set of $H$ to that of $G$ (by homomorphisms, we mean that for every edge in $H$, the corresponding image should also be found in $G$, but not necessarily the other way around). The \emph{Homomorphism density} of $H$ in $G$ is then defined as
$$
t(H,G) = \frac{ \mathrm{Hom}(H,G)}{N^m}.
$$

Let $H_1,\dots,H_l$ be finite simple graphs and $\beta_1,\dots,\beta_l$ be real numbers. Let $G$ be a random simple graph on $N$ vertices defined by
\begin{equation}\label{eq:defexpgraph}
\PP(G = g) = Z^{-1} \exp \left (N^2 \sum_{i=1}^l \beta_i t(H_i,g) \right )
\end{equation}
for all simple graph $g$ on $n$ vertices, where $Z$ is a normalizing constant. The graph $G$ is referred to in the literature as an exponential random graph (see e.g. \cite{CDExponential} and references therein).

For $\vec{p} \in [0,1]^{{N \choose 2}}$ define by $G(N,\vec{p})$ the random graph whose edges determined by independent Bernoulli random variables whose probabilities correspond to the vector $\vec p$. Define 
$$
I(\vec p) = - \sum_{(i,j) \in {N \choose 2}} \Bigl ( p_{i,j} \log p_{i,j} + (1-p_{i,j}) \log (1-p_{i,j}) \Bigr ) = \Ent(G(N, \vec{p})).
$$
Moreover, for a probability measure $\rho$ on $[0,1]^{{N \choose 2}}$ define by $G(N, \rho)$ the "$\rho$-mixture" satisfying
$$
\PP(G(N,\rho) = g) = \int  \PP(G(N,\vec p) = g) d \rho(\vec p)
$$
for all simple graphs $g$ on $N$ vertices.

It is clear that every random graph $G$ on $N$ vertices has the distribution $G(N,\rho)$ for some measure $\rho$ (we can simply take $\rho$ to be supported on $\{0,1\}^{{N \choose 2}}$ with probabilities corresponding to the individual instances). However, it is interesting to look for a representation where most of the entropy comes from the graphs $G(N,\vec p)$ rather than from the distribution $\rho$. We thus make the following definition.
\begin{definition}
	We say that a random graph $G$ is an $\eps$-mixture if there exists a measure $\rho$ on $[0,1]^{{N \choose 2}}$ such that $G = G(N,\rho)$ and such that
	$$
	\mathrm{Ent}(G) \leq \int I(\vec p) d \rho(\vec p) + \eps {N \choose 2} = \int \mathrm{Ent}(G(N, \vec p)) d \rho(\vec p) + \eps {N \choose 2}.
	$$
\end{definition}

Finally, for two simple graphs $g=(V,E),g'=(V,E')$ define $d_H(g,g') = |E \Delta E'|$, the Hamming distance between the corresponding edge sets. Our theorem roughly says that exponential graphs can be coupled with $o(1)$-mixtures in a way that the Hamming distance is $o\left ( {N \choose 2} \right )$.
\begin{theorem} \label{thm:exponential}
	For any integers $N,l$, finite simple graphs $H_1,\dots,H_l$, real numbers $\beta_1,\dots,\beta_l$ and $\eps \in (0, 1/2)$, there exists a coupling $(G,G')$ such that the marginal $G$ is the associated random exponential graph defined in equation \eqref{eq:defexpgraph}, the graph $G'$ is an $\eps$-mixture, and such that
	$$
	\EE d_H( G, G' ) = \frac{20 {N \choose 2}^{11/12}}{\eps^{1/3}} \left ( \sum_{i=1}^l |\beta_i| |E(H_i)| \right )^{1/3}.
	$$ 
	where $E(H_i)$ denotes the number of edges of $H_i$.
\end{theorem}

\begin{remark}
The ideas results of this section are extended in a subsequent work \cite{EldanGross-Expographs}, where it is shown that, in the dense regime, the measure $\rho$ is essentially supported on block matrices.
\end{remark}

\subsection{Approach} \label{sec:approach}
Our two main theorems, Theorem \ref{thm:maindecomp} and Theorem \ref{thm:mainGaussian}, heavily rely on a construction coming from stochastic control theory, of an entropy-optimal coupling of the measure $\nu$ to a Brownian motion \cite{Follmer85, Borell02, Lehec-13}, described below. This coupling has proven to be a strong tool for proving functional inequalities: In \cite{Borell02} it is used to give a proof of the Pr\'ekopa-Leindler inequality. Later on, in \cite{Lehec-13} a representation formula for the relative entropy was derived which can be used to provide extremely simple proofs of several information-theoretical inequalities on Gaussian space, such as Shannon's inequality and the Log-Sobolev inequality. In \cite{EL-regularization} the same coupling was used to prove an $L_1$-version of hypercontractivity on Gaussian space, resolving the Gaussian variant of a conjecture by Talagrand whereas in \cite{ELL15} it is used to show that a local-curvature condition implies a transportation-entropy inequality for Markov chains.

Let us now describe this coupling and the general lines in which it is used to prove Theorem \ref{thm:maindecomp}. Fix a measure $\nu$ on $\DC$ with $d \nu = e^f d \mu$. Define $\CC=[-1,1]^n$, the convex hull of $\DC$. Let $B_t$ be a standard Brownian motion on $\RR^n$ and let $X_t$ be a process satisfying, for all $i \in [n]$, $d \langle X_t, e_i \rangle = \mathbf{1}\left \{ |\langle X_t, e_i \rangle| < 1 \right \} d \langle B_t, e_i \rangle$. In other words, $X_t$ is a Brownian motion such that whenever a facet of $\CC$ is hit, the corresponding coordinate stops moving. We will thus have that $X_\infty$ has the law $\mu$, the uniform measure on $\DC$. The idea is to introduce a \emph{change of measure} on the path space, which reweighs every path of $X_t$ according to the value $\exp(f(X_\infty))$. In other words, if $P$ was the original measure on Weiner space according to which $B_t$ was a Brownian motion, we consider a new measure $Q$ such that $\frac{dQ}{dP} \propto \exp(f(X_\infty))$.

A-priori, it is not even clear whether under this reweighing, the process $X_t$ is Markovian. However, as it turns out, this reweighing has an alternative interpretation in terms of \emph{drift}. Namely, it turns out that under the measure $Q$,  
$$
d X_t = d \tilde B_t + v_t dt 
$$ 
where $\tilde B_t$ is a Brownian motion with respect to $Q$ and $v_t$ is an adapted drift (the formula is valid as long as $X_t$ is still in the interior of $[-1,1]^n$). The drift $v_t$ turns out to be \emph{entropy-minimizing} in the following sense: remark that by definition, since we reweighed every path according to the value of $f$ at the endpoint, we have that $\KL(\nu \Vert \mu) = \KL(Q \Vert P)$. In other words, the relative entropy between the distribution of the whole path of $X_t$ and that of a Brownian motion is equal to the relative entropy between the endpoints, which roughly means that the $v_t$ has to minimize the relative entropy at every infinitesimal step. As shown in \cite{Lehec-13}, among the drifts $v_t$ under which we have $X_\infty \sim \nu$, the drift $v_t$ is the one minimizing $\EE \int_0^\infty |v_t|^2 dt$. An easy consequence of this is that $v_t$ has to be a martingale (up to the fact that it becomes zero in coordinates that reach $\{-1,+1\}$). Moreover, a calculation gives that
$$
v_\infty \approx \nabla f(X_\infty).
$$
The fact that $v_t$ is a martingale tells us that $v_t = \EE[v_\infty | X_t]$, which means that $v_t$ is always approximately inside the convex hull of $\{\nabla f(y): y \in \DC \}$. Thus, our complexity assumption amount to the fact that the drift $v_t$ is "trapped" inside a small set. Another useful consequence of the entropy-minimizing property of $v_t$ is that $d v_t = \Gamma_t d B_t$ for a matrix $\Gamma_t$ which dominates the matrix $\mathrm{Cov}(v_\infty | X_t)$ (in a positive-definite sense). Roughly speaking the latter tells us that if $v_t$ is expected to change significantly by time $\infty$, then it must start moving right away (or to put this property in yet simpler words, if $v_t$ needs to make a choice at some point, it will try to make this choice asap).

The latter property of $v_t$, which tells us that as long as it has some variance left, it is moving quickly, combined with the property that it needs to be trapped in a set of small Gaussian width tell us that $v_t$ must roughly stop moving by a time $t$ which is not too big. This fact, which is at the heart of the proof is the content of Lemma \ref{lem:divergence} below. 

Finally, once the drift $v_t$ is roughly deterministic, then the distribution of $X_\infty | X_t$ becomes close to a product measure (it is easily seen that if $v_t$ is deterministic then $d X_t$ has independent coordinates). Thus, the relevant decomposition in our theorem \ref{thm:maindecomp} will be to the measures defined by the law of $X_\infty | X_t$. When the time $t$ is small, the corresponding measures will be nothing but small tilts of the original measure $\nu$.

In Section \ref{sec:idea} we also demonstrate how one can derive a bound of the spirit of Theorem \ref{thm:main} for the Gaussian case without using the above coupling. This produces a weaker result but may give a better intuition for the way that the low complexity is used.

\medskip {\emph{ Acknowledgements.}} I would like to thank Sourav Chatterjee for the wonderful series of lectures about the topic of nonlinear large deviations which inspired me to work on this topic (the lectures were given at the Texas A\&M University Concentration Week on Geometric Functional Analysis organized by Johnson, Paouris and Rudelson). I'm also thankful to Ramon Van Handel, Amir Dembo, Yair Shenfeld and James Lee for enlightening discussions and suggestions. Finally, I thank the anonymous referee for her extremely useful comments and suggestions which have significantly improved the presentation of this paper.

\section{The Gaussian case}

\subsection{Tilts are close to product measures} \label{sec:idea}
The goal of this subsection is to illustrate a general idea of how Gaussian width can produce tilts which are close to Gaussian in transportation distance. The proof in this section is rather straightforward and demonstrates the way in which Gaussian width complexity comes to play: we define a suitable vector field on $\RR^n$ (which plays the same role of the drift $v_t$ defined in Section \ref{sec:approach}) which is restricted to be inside the convex hull of $\{\nabla f(x): x \in \RR^n\}$. Then, the assumption on the Gaussian width, via an application of the divergence theorem, implies a bound on the divergence of this vector field at a point, which in turn implies the existence of a product-like tilt. \\

Recall that $\gamma$ denotes the standard Gaussian measure on $\RR^n$. Let $f:\RR^n \to \RR$ be twice differentiable and $\nu$ satisfy $\frac{d \nu}{d \gamma} = e^{f}$. 
Now, for all $x \in \RR^n$, let us consider the measure $\nu_x$ defined by
$$
d \nu_x(y) = \frac{\exp(\langle x, y \rangle) d \nu(y)}{\int \exp(\langle x, z \rangle) d \nu(z)} = \frac{\exp(f(y) + \langle x, y \rangle )}{ \int \exp(f(z) + \langle x, z \rangle) d \gamma(z) } d \gamma(y)
$$
When the vector $x$ is small we can think of the measure $\nu_x$ as a "tilt" of the measure $\nu$ towards the direction $x$. We prove the following,
\begin{theorem} \label{thm:gausstilt}
Let $\nu$ be a measure on $\RR^n$ satisfying $d \nu = e^f d \gamma$ for a twice-differentiable function $f$, and define $\nu_x$ as above. Then for every $r>0$, there exists a point $x_0 \in \RR^n$ with $|x_0| \leq r$, such that
$$
\WW(\nu_{x_0}, \gamma_{u})^2 \leq 2 \frac{\sqrt{n}}{r} \Comp(f) - \inf_{y \in \RR^n} \Delta f(y)
$$
where $\gamma_u$ is the Gaussian whose centroid is equal to the centroid of $\nu_{x_0}$ and having identity covariance.
\end{theorem}
Define also
$$
h(x) = \int \exp(\langle x, y \rangle - |x|^2/2) d \nu(y) = (2 \pi)^{-n/2} \int \exp(f(y) + \langle x, y \rangle - |y|^2/2 - |x|^2/2) dy
$$
$$
= \int \exp(f(y+x)) d \gamma(y)
$$
Next, we consider the vector field
\begin{equation}\label{eq:defvx}
v(x) = \nabla \log h(x) = \frac{\int \nabla f(x+y) \exp(f(y+x) - |y|^2/2) dy}{\int \exp(f(y+x) - |y|^2/2) dy} = \int \nabla f(y) d \nu_x(y).
\end{equation}
A straightforward calculation gives that
$$
\nabla v(x) = \frac{\int \left (\nabla f(x+y)^{\otimes 2} + \nabla^2 f(x+y) \right ) \exp(f(y+x) - |y|^2/2) dy}{\int \exp(f(y+x) - |y|^2/2) dy} - v(x)^{\otimes 2}
$$
\begin{equation}\label{eq:dvx}
= \int \left (\nabla f(y)^{\otimes 2}  + \nabla^2 f(y) \right ) d \nu_x(y) - \left ( \int \nabla f(y) d \nu_x(y) \right )^{\otimes 2}.
\end{equation}
Define $K$ to be the convex hull of the set $\{\nabla f(y): ~ y \in \RR^n \}$. By equation \eqref{eq:defvx}, it is evident that $v(x) \in K$ for all $x \in \RR^n$.

Now, fix a parameter $r>0$ and define by $\omega_n$ the $n-1$-dimensional Hausdorff measure of $\Sph$. By standard estimates concerning the norm of a Gaussian random variable, we have that $\EE[|\Gamma|] \geq \sqrt{n}/2$ and therefore
$$
\MW(K) = \EE \left [ \sup_{x \in K} \langle x, \Gamma \rangle \right ] \geq \frac{\sqrt{n}}{2} \frac{1}{ r^{n-1} \omega_n } \int_{r \Sph} \sup_{x \in K} \langle x, \theta / r \rangle d \mathcal{H}_{n-1}(\theta)
$$
where $\mathcal{H}_{n-1}$ denotes the $(n-1)$-dimensional Hausdorff measure. It follows that, in particular
\begin{equation}\label{eq:meanw1}
\frac{1}{\omega_n r^{n-1}} \int_{r \Sph} \langle \theta / r, v(\theta) \rangle d \mathcal{H}_{n-1}(\theta) \leq \frac{2}{\sqrt{n}} \MW(K).
\end{equation}
On the other hand, by the divergence theorem we have
$$
\int_{|x| \leq r} \tr(\nabla v(y)) dy = \int_{r \Sph} \langle \vec{n}_\theta, v(\theta) \rangle d \mathcal{H}_{n-1}(\theta)
$$
where $\vec{n}_\theta$ denotes the outer unit normal to $r \Sph$ at $\theta$. Combining the two last inequalities, together with the identity $\mathrm{Vol}(\{|x| \leq r \}) = \tfrac{1}{n}r^{n} \omega_n$, yields
$$
\frac{1}{\mathrm{Vol}(\{|x| \leq r \})} \int_{|x| \leq r} \tr(\nabla v(y)) dy = \frac{n}{r^{n} \omega_n} \int_{r \Sph} \langle \vec{n}_\theta, v(\theta) \rangle d \mathcal{H}_{n-1}(\theta) \leq 2 \frac{\sqrt{n}}{r} \MW(K).
$$
Consequently, there exists a point $x_0$ with $|x_0| \leq r$ such that
\begin{equation}\label{eq:trsmall1}
\tr(\nabla v(x_0)) \leq 2 \frac{\sqrt{n}}{r} \MW(K).
\end{equation}

Using equation \eqref{eq:dvx}, we have that
\begin{align}
\tr \left (\nabla v(x) \right ) ~& = \int \left ( \Delta f(y) + |\nabla f(y)|^{2} \right ) d \nu_x(y) - |v(x)|^2  \nonumber \\
& \geq \int |\nabla f(y)|^2 d \nu_x(y) - |v(x)|^2 + \inf_{y \in \RR^n} \Delta f(y). \label{eq:trnablav}
\end{align}

Define now a measure $\gamma_x$ by $d \gamma_x = e^{\langle y,x \rangle - |x|^2/2} d \gamma(y)$. Moreover, for a measure $\rho$ consider the Fisher information of $\rho$ with respect to $\gamma_x$, defined as
$$
\I_{\gamma_x}(\rho) = \int \left |\nabla \log \frac{d \rho}{d \gamma_x} \right |^2 d \rho.
$$
Then we have, by definition of $\nu_x$,
\begin{align*}
\I_{\gamma_x} (\nu_x) ~& = \int \left  |\nabla \log \frac{d \nu_x}{d \gamma_x} \right |^2 d \nu_x  = \int \left  |\nabla f(y) \right |^2 d \nu_x(y).
\end{align*}
Combining \eqref{eq:trsmall1}, \eqref{eq:trnablav} and the above identity finally gives
$$
\I_{\gamma_{x_0}} (\nu_{x_0}) \leq 2 \frac{\sqrt{n}}{r} \MW(K) - \inf_{y \in \RR^n} \Delta f(y) + |v(x_0)|^2.
$$

Recall the transportation-entropy inequality of Talagrand \cite{Talagrand96}, which states that
$$
\WW(\rho, \gamma)^2 \leq 2 \KL(\rho \Vert \gamma).
$$
for every measure $\rho$ for which the right hand side is defined. Combined with the log-Sobolev inequality $2 \KL(\rho \Vert \gamma) \leq \I_{\gamma}(\rho)$, we have that
\begin{equation}\label{eq:finaltransp}
\WW(\nu_{x_0}, \gamma_{x_0})^2 \leq 2 \frac{\sqrt{n}}{r} \MW(K) - \inf_{y \in \RR^n} \Delta f(y) + |v(x_0)|^2.
\end{equation}
Finally, remark that by the definition of $h(x)$ and by integration by parts we have
$$
\nabla h(x) = (2 \pi)^{-n/2} \int \nabla \left (e^{f(x+y)} \right ) e^{-|y|^2/2} d y = (2 \pi)^{-n/2} \int y e^{f(x+y)} e^{-|y|^2/2} d y
$$
$$
= e^{-|x|^2/2} \int (y-x) e^{f(y) + \langle x, y \rangle} d \gamma(z)
$$
Which implies the identity
$$
v(x) = \int (y-x) d \nu_x(y).
$$
Now, since for every pair of random vectors $X,Y \in \RR^n$ one has the parallelogram identity $\EE |X-Y|^2 = \EE |(X - \EE[X]) - (Y - \EE[Y])|^2 + |\EE[X]-\EE[Y]|^2$, we have that
$$
\WW(\nu_{x_0}, \gamma_{x_0})^2 = \left |\int y d \nu_{x_0}(y) - \int y d \gamma_{x_0}(y) \right |^2 + \WW(\nu_{x_0}, \gamma_{u})^2 = |v(x_0)|^2 + \WW(\nu_{x_0}, \gamma_{u})^2
$$
where $u = \int y d \nu_{x_0}(y) = x_0+v(x_0)$ is the centroid of $\nu_{x_0}$. Together with equation \eqref{eq:finaltransp}, we get
$$
\WW(\nu_{x_0}, \gamma_{u})^2 \leq 2 \frac{\sqrt{n}}{r} \MW(K) - \inf_{y \in \RR^n} \Delta f(y).
$$
This finishes the proof of Theorem \ref{thm:gausstilt}.

\subsection{A reverse log-Sobolev inequality}

In this subsection we prove theorem \ref{thm:mainGaussian}. Fix the function $f$ and the measure $\nu$, such that $d \nu = e^f d \gamma$. Assume that $f$ is 
twice differentiable. Our proof is based on the following stochastic construction, for which we make similar definitions as in \cite{EL-regularization}. Let $B_t$ be a standard Brownian motion in $\RR^n$ adapted to a filtration $\FF_t$. Consider the Ornstein-Uhlenbeck convolution operator
$$
P_t[g](x) = \frac{1}{(2 \pi t)^{n/2} } \int_{\RR^n} g(y) \exp \left ( - \frac{|x-y|^2}{2t} \right ) dy = \EE[g(x+B_t)].
$$
Define 
$$
Z(t,x) = P_{1-t}[e^f](x),
$$
and for $t \in (0,1)$ and $x \in \RR^n$ consider the measure defined by
$$
\nu_{t,x}(A) = \frac{P_{1-t}[e^f \mathbf{1}_A](x)}{ P_{1-t}[e^f](x)} = \frac{\EE \left [ \mathbf{1}_{\{B_1 \in A\}} e^{f(B_1)} | B_t = x \right ]  }{\EE \left [ e^{f(B_1)} | B_t = x \right ]  }
$$
for every measurable $A \subset \RR^n$. Define also $\nu_{1,x} = \delta_x$, a Dirac measure supported on $x$. Remark that $\nu_{0,0} = \nu$. Finally consider the vector field
$$
v(t,x) = \nabla_x \log Z(t,x) = \frac{\nabla P_{1-t}[e^f](x)}{ P_{1-t}[e^f](x) } = \frac{P_{1-t}[\nabla f e^f ](x)  }{ P_{1-t}[e^f](x) } = \int \nabla f(y) d \nu_{t,x}(y).
$$
Integration by parts yields that
$$
v(t,x) = \frac{ \int \nabla f(y) e^{f(y) - \tfrac{|x-y|^2}{2(1-t)}} dy } { \int e^{f(y) - \tfrac{|x-y|^2}{2(1-t)}} dy} = (1-t)^{-1} \frac{ \int (y-x) e^{f(y) - \tfrac{|x-y|^2}{2(1-t)}} dy } { \int e^{f(y) - \tfrac{|x-y|^2}{2(1-t)}} dy } = (1-t)^{-1} \int (y-x) d \nu_{t,x}(y).
$$
We also have,

$$
\nabla_x v(t,x) = \nabla^2 \log Z(t,x) = \frac{\nabla^2 Z(t,x)}{Z(t,x)} - \left (\frac{\nabla Z(t,x)}{Z(t,x)} \right )^{\otimes 2}
$$
$$
= \frac{\nabla^2 P_{1-t} [e^f](x)  }{ P_{1-t} [e^f](x) } - v(t,x)^{\otimes 2} = \frac{ P_{1-t} [ \left ( \nabla^2 f + \nabla f^{\otimes 2} \right ) e^f](x) } { P_{1-t}[e^f](x)} - v(t,x)^{\otimes 2}
$$

which gives the formula
\begin{equation}\label{eq:nablavtx}
\nabla_x v(t,x) = \int \left  ( \nabla^2 f(y) + \nabla f(y)^{\otimes 2}  \right ) d \nu_{t,x}(y) - v(t,x)^{\otimes 2} =: \Gamma(t,x).
\end{equation}
Consider now the process $X_t$ which solves the stochastic differential equation
$$
X_0 = 0, ~~ d X_t = d B_t + v_t dt.
$$
where we define
$$
v_t := v(t,X_t).
$$
The following facts are proven, for instance, in \cite[Section 2.2]{EL-regularization}. The representation formula, equation \eqref{eq:representation} below was shown in \cite{Lehec-13}.
\begin{lemma} \label{lem:basicprops}
The processes $X_t, v_t$ have the following properties
\begin{enumerate}[label=(\roman*)]
\item \label{item:girsanov1}
The random variable $X_1$ has the law $\nu$, and for any time $t$ one has almost surely that $X_1| \FF_t$ has the law $\nu_{t,X_t}$.
\item
The process $v_t$ is a martingale.
\item
The relative entropy of $\nu$ can be expressed as
\begin{equation}\label{eq:representation}
\EE \int_0^1 |v_t|^2 dt = 2 \KL(\nu \Vert \gamma).
\end{equation}
\item
The process $v_t$ satisfies
\begin{equation}\label{eq:dvt22}
d v_t = \Gamma(t, X_t) d B_t.
\end{equation}
\end{enumerate}
\end{lemma}

An immediate corollary is
\begin{fact} \label{fact:vtink}
For all $t \in [0,1]$, we have almost surely
\begin{equation}\label{eq:vtinconvhull}
v_t \in \mathrm{Conv}(\nabla f(y): y \in \RR^n). 
\end{equation}
\end{fact}
\begin{proof}
Since $v_t$ is a martingale we have $v_t = \EE[v_1 | \FF_t] = \EE[\nabla f(X_1) | \FF_t]$.
\end{proof}
Defining $\Gamma_t = \Gamma(t,X_t)$, we have by \ref{item:girsanov1} in the above lemma that
$$
\Gamma_t = \EE[v_1^{\otimes 2} + \nabla^2 f(X_1) | \FF_t] - v_t^{\otimes 2}.
$$
Set $H_t = \EE[(v_1 - v_t)^{\otimes 2} | \FF_t]$. Since $v_t$ is a martingale, we have that 
$$
H_t = \EE[v_1^{\otimes 2} | \FF_t] - v_t^{\otimes 2}.
$$
By It\^{o}'s isometry and by formula \eqref{eq:dvt22}, we have that 
$$
\EE \left [v_s^{\otimes 2} - v_t^{\otimes 2} \right ] = \EE \left [\int_s^t \Gamma_r^2 dr \right ], ~~ \forall 0 \leq t \leq s \leq 1.    
$$
Since $\Gamma_r^2$ is positive semi-definite, a combination of the two last displays gives
$$
\EE \tr (H_t) \leq \EE \tr (H_s)
$$
for all $0 \leq s \leq t \leq 1$. We get that
\begin{equation}\label{eq:Gammatineq}
\EE \tr (H_t) - M \leq \EE \tr (\Gamma_s), ~~ \forall 0 < s < t < 1
\end{equation}
where we define
$$
M := - \inf_{x \in \RR^n} \Delta f(x).
$$
The following lemma is the central place where the Gaussian-width functional plays a role. 
\begin{lemma} \label{lem:mwGauss}
Define $K = \mathrm{Conv}(\{ \nabla f(y): y \in \RR^n\} )$. For every $t \in (0,1)$ we have
$$
\EE [\tr(H_t)] \leq \frac{\MW(K)}{ \sqrt{t}} + M.
$$
\end{lemma}
\begin{proof}
We have by It\^{o}'s isometry and by \eqref{eq:dvt22} that, for all $t \in (0,1)$,
$$
\EE \left [ \langle B_{t}, v_{t} \rangle \right ] = \EE \left [ \int_0^{t} \tr(\Gamma_s) ds \right] \stackrel{ \eqref{eq:Gammatineq}}{\geq} t \left (\EE \tr(H_t) - M \right ).
$$
On the other hand, since almost surely $v_t \in K$ by Fact \ref{fact:vtink}, we have
$$
\EE \left [ \left  \langle \frac{B_t}{\sqrt t},  v_t \right \rangle \right  ] \leq \EE \left [ \sup_{x \in K} \left \langle \frac{B_t}{\sqrt t}, x \right \rangle \right ] = \MW(K). 
$$
Combining the two above inequalities, we get
$$
\EE [\tr(H_t)] \leq \frac{\MW(K)}{ \sqrt{t}} + M
$$
which completes the proof.
\end{proof}

We are finally ready to prove the reverse log-Sobolev inequality.
\begin{proof}[Proof of Theorem \ref{thm:mainGaussian}]
Fix a time $t \in (0,1)$. Since $v_t$ is a martingale, we have that $\EE |v_t|^2 \leq \EE |v_s|^2$ for all $t \leq s$, which gives
$$
2 \KL(\nu \Vert \gamma) \stackrel{\eqref{eq:representation}}{=} \EE \int_0^1 |v_s|^2 ds \geq  \EE \int_t^1 |v_s|^2 ds \geq (1-t) \EE |v_t|^2.
$$
On the other hand, by Lemma \ref{lem:basicprops} we also have
$$
\I(\nu) = \EE |\nabla f(X_1)|^2 = \EE |v_1|^2 = \EE \left [ \tr(H_t) + |v_t|^2 \right ]
$$
and using Lemma \ref{lem:mwGauss} we have
$$
\I(\nu) \leq \EE |v_t|^2 + \frac{\MW(K)}{ \sqrt{t}} + M.
$$
Combining the above inequalities, we have 
$$
\I(\nu) - 2 \KL(\nu \Vert \gamma) \leq t \EE |v_t|^2 + \frac{\MW(K)}{ \sqrt{t}} + M \leq t \I (\nu) + \frac{\MW(K)}{ \sqrt{t}} + M
$$
Now, if $\MW(K) \leq \I(\nu)$ then taking $t = \left (\frac{\MW(K)}{\I(\nu)}\right )^{2/3}$ gives
$$
\I(\nu) - 2 \KL(\nu \Vert \Gamma) \leq 2 \MW(K)^{2/3} \I(\nu)^{1/3} + M
$$
Otherwise if $\MW(K) > \I(\nu)$, we trivially have that $\I(\nu) - 2 \KL(\nu \Vert \Gamma) \leq 2 \MW(K)^{2/3} \I(\nu)^{1/3}$. The proof is complete.
\end{proof}

\section{The discrete case}
Our main goal in this section is to prove Theorem \ref{thm:maindecomp}. We then show that Theorem \ref{thm:main} follows easily. A crucial element of our approach is to consider harmonic extensions of the function $f$ and related functions into the continuous cube, which we do in subsection \ref{sec:preliminaries}. The core idea of the proof is based on a stochastic construction defined in subsection \ref{sec:stochastic}.

\subsection{Some preliminary definitions} \label{sec:preliminaries}
\subsubsection{Harmonic extensions}
We define $\CC = [-1,1]^n$, the convex hull of $\DC$. In the following, we will use the notation $\nabla$ to denote both a discrete and a continuous gradient, depending on the domain of the function. From here on, for the sake of brevity, the notation $y$ will usually be used for points in $\DC$ while $x$ will be used for points in $\CC$. Denote by $e_1,\dots,e_n$ the vectors of the standard basis on $\RR^n$.

For $x \in [-1,1]$ and $y \in \{-1,1\}$, define
$$
w(x,y) = \frac{1+xy}{2}
$$
so that for all $x \in [-1,1]$, $w(x, \cdot)$ is a probability density on $\{-1,1\}$ of a measure whose expectation is $x$. By slight abuse of notation, for $x \in \CC$ and $y \in \DC$, we write
$$
w(x,y) = \prod_i w(x_i, y_i).
$$
For a function $\tilde \xi:\DC \to \RR$, the harmonic extension to $\CC$ is the function defined by the equation
$$
\xi(x) = \sum_{y \in \DC} w(x,y) \tilde \xi(y).
$$
This is the unique function satisfying the following three conditions: (i) it is harmonic in the interior of $\CC$, (ii) for each $k$-facet of $\CC$, it is harmonic inside the relative interior of this facet with respect to the $k$-Laplacian associated with the corresponding affine subspace and (iii) it coincides with $f$ on $\DC$.

We have the following easy fact.
\begin{fact} \label{fact:harmonicext}
	If $\xi(x)$ is the harmonic extension of $\tilde \xi : \DC \to \RR$ to $\CC$ then $\partial_i \xi$ is the harmonic extension of $\partial_i \tilde \xi$, or in other words
	\begin{equation}\label{eq:nablaharmonic}
	\nabla  \xi (x) = \sum_{y \in \DC} w(x,y) \nabla \tilde \xi(y), ~~ \forall x \in \CC.
	\end{equation}
\end{fact}
\begin{proof}
	Suppose first that $\tilde \xi (y) = \mathbf{1}_{ \{ y = \epsilon\} }$	for some $\epsilon \in \DC$. Then $\xi(x) = 2^{-n} \prod_j (1 + \epsilon_j x_j)$, which implies that $\partial_i \xi(x) = 2^{-n} \epsilon_i \prod_{j \neq i} (1 + \epsilon_j x_j)$. On the other hand $\partial_i \tilde \xi(y) = \tfrac{1}{2} \epsilon_i \mathbf{1}_{ \{ y_{-i} = \epsilon_{-i} \}}$, where the notation $v_{-i}$ stands for the vector $v$ with the $i$-th coordinate omitted. It is now straightforward to check that $\partial_i \xi(x)$ is indeed the harmonic extension of $\partial_i \tilde \xi$. The proof is concluded due to the linearity of both sides of \eqref{eq:nablaharmonic} with respect to $\tilde \xi$.
\end{proof}
As a consequence, we have the following simple but useful result.
\begin{fact} \label{fact:extmartingale}
	Let $\tilde \xi: \DC \to \RR$ and let $\xi$ be the harmonic extension of $\tilde \xi$ to $\CC$. Then for any diagonal matrix $A$ and for all $x \in \CC$, one has that $\tr(A \nabla^2 \xi(x)) = 0$. Consequently, if $Y_t$ is a martingale taking values in $\CC$ such that $d Y_t = \sigma_t d B_t$, where $B_t$ is a Brownian motion and $\sigma_t$ is almost surely diagonal for all $t$, then the process $\xi(Y_t)$ is also a martingale.
\end{fact}
\begin{proof} Use Fact \ref{fact:harmonicext} to conclude that $\nabla^2 \xi$ has zeroes on its diagonal. It follows from It\^{o}'s formula that $\xi(Y_t)$ is a local martingale. Since $\xi(\cdot)$ is bounded, we conclude that $\xi(Y_t)$ is a martingale.
\end{proof}

\subsubsection{Some core constructions}
Let $\nu$ be a probability measure on $\DC$. Define $f_{\nu}(y) = \log \frac{d \nu}{d \mu}(y)$ for all $y \in \DC$. In the following, we abbreviate $f=f_\nu$ whenever there is no ambiguity. For a point $y \in \DC$ define 
\begin{equation}\label{eq:defvnu}
v(y) = v_\nu(y) = \frac{\nabla e^{f}(y)}{e^{f(y)}}
\end{equation}
(using the definition of the discrete gradient $\nabla$) with the convention $v_\nu(y)=0$ when $\frac{d \nu}{d \mu}(y)=0$. Note that the identity $\frac{\nabla e^{f}(y)}{e^{f(y)}} = \nabla f(y)$ is \textbf{not} true in the discrete setting, but the reader can assume that it is approximately correct for the sake of intuition. The purpose of some of our definitions below is to overcome this caveat, see Remark \ref{rmkder} below. \\

Let $h_\nu(x)$ be the harmonic extension of the function $e^f$ to $\CC$ or in other words,
$$
h(x) = h_\nu(x) = \sum_{y \in \DC} w(x,y) e^{f(y)}.
$$
We extend the function $v_\nu$ from $\DC$ to $\CC$ by defining
\begin{equation}\label{eq:defvx2}
v_\nu(x) = \frac{\nabla h(x)}{h(x)}, ~~ \forall x \in \CC
\end{equation}
with the convention $v_\nu(x) = 0$ when $h(x)=0$. Remark that $v_\nu(x)$ is not harmonic in general, however Fact \ref{fact:harmonicext} implies that the latter definition is in accordance with equation \eqref{eq:defvnu} in the sense that the two definitions coincide on $\DC$.

For $x \in [-1,1]$ and $g \in (-1,1)$, consider the function
$$
\zeta_x(g) = \frac{ g }{ 1 + g x }
$$
and its inverse
$$
\zeta_x^{-1} (v) = \frac{v}{1 - v x}. 
$$
With slight abuse of notation, for $g=(g_1,...,g_n) \in \CC$ and $x=(x_1,...,x_n) \in \CC$  we define
$$
\zeta_x (g) = (\zeta_{x_1}(g_1), \dots , \zeta_{x_n} (g_n) )
$$
and define $\zeta^{-1}_x$ likewise. The point of this definition will be clarified later, but a useful way to understand $\zeta$ is the fact that if $v=\zeta_x(g)$ then
\begin{equation}\label{eq:gv}
g = \frac{b-a}{b+a} \Leftrightarrow v = \frac{b-a}{(1+x)b+(1-x)a}. 
\end{equation}
Both quantities above should be thought of as discrete interpretations the quantity $\nabla \log \xi$ for a function $\xi:\{-1,1\} \to (0,\infty)$ satisfying $\xi(1) = b, \xi(-1) = a$. 

Finally, we will define the function
\begin{equation}\label{eq:defgnu}
g_\nu(y) = \zeta_y^{-1}(v_\nu(y)), ~~ \forall y \in \DC
\end{equation}
In other words for $i \in [n]$, $y \in \DC$ if we denote $y_+,y_-$ to be the points satisfying $\langle y_+, e_j \rangle = \langle y_-, e_j \rangle = \langle y, e_j \rangle$ for all $j \neq i$ and $\langle y_{\pm}, e_i \rangle = \pm 1$, then
\begin{equation}\label{eq:defgp}
\langle g_\nu(y), e_i \rangle = \frac{e^{f(y_+)} - e^{f(y_-)} }{ e^{f(y_+)} + e^{f(y_-)} } = \left \langle v_\nu \left ( \frac{y_+ + y_-}{2} \right ), e_i \right \rangle.
\end{equation}

\begin{remark} \label{rmkder}
Both the quantities $g_\nu(y)$ and $v_\nu(y)$ should be thought of as approximations of $\nabla f(y)$. We will need to distinguish between those approximations because the chain rule $\nabla e^f = \nabla f e^f$ does not hold true in the discrete setting. Note also that
\begin{equation}\label{eq:tanh}
g_\nu(y) = \tanh(\nabla f_\nu(y)).
\end{equation}
\end{remark}
Next, remark that 
\begin{equation}\label{eq:gcom}
\int g_{\nu}(y) d \nu(y) = \int y d \nu(y)
\end{equation}
which is a consequence of the calculation
$$
\left \langle \int g_{\nu}(y) d \nu(y), e_i \right \rangle = \sum_{y \in \DC} e^{f(y)} \langle g_\nu(y), e_i \rangle = 2 \sum_{y \in \DC} \partial_i e^f (y) = \sum_{y \in \DC} y_i e^{f(y)}.
$$
A central definition in our proofs will be the matrix
$$
\HH(\nu) := \int_{\DC} g_{\nu}(y)^{\otimes 2} d \nu(y)  - \left ( \int_{\DC} g_{\nu}(y) d \nu(y) \right )^{\otimes 2},
$$
the covariance matrix of the random vector $g_\nu(X)$ for $X \sim \nu$. \\

The following lemma is a straightforward application of the Sudakov-Fernique inequality:
\begin{lemma} \label{lem:SF}
We have, in the above notation,
\begin{equation}\label{eq:SF}
\MW \left ( \left \{ g_\nu(y) : y \in \DC  \right \} \right ) \leq \Comp(f_\nu) = \Comp(\nu).
\end{equation}
\end{lemma}
\noindent The proof is postponed to the appendix. \\

Finally, for the sake of intuition, let us calculate the field $v_\nu(x)$ for the case that $\nu$ is a tilt of the uniform measure, $\nu = \tilt_\theta(\mu)$. In this case, we have 
$$
h(x) = C_\theta \sum_{y \in \DC} w(x,y) \exp(\langle \theta, y \rangle) = C_\theta \prod_{i=1}^n \left ( \frac{1+x_i y_i}{2} e^{\theta_i} + \frac{1 - x_i y_i}{2} e^{-\theta_i} \right )
$$
where $C_\theta$ is a normalization constant. Therefore
\begin{align*}
v_\nu(x) &~ = \nabla \log h(x) = \sum_{i=1}^n \nabla \log \left ( \frac{1+x_i y_i}{2} e^{\theta_i} + \frac{1 - x_i y_i}{2} e^{-\theta_i} \right ) \\
& = \sum_{i=1}^n \frac{e^{\theta_i} - e^{-\theta_i}}{(1+x_i) e^{\theta_i} + (1-x_i) e^{-\theta_i}} \mathrm{e}_i.
\end{align*}
Remark that $g_\nu(y) = v_\nu(0) = \tanh(\theta)$ for all $y \in \DC$, therefore in this case we have $\HH(\nu) = 0$. We will later see that this is robust in the sense that whenever the matrix $\HH(\nu)$ is small, the measure $\nu$ is close to a product measure.

\subsection{Two main steps towards the proof} \label{sec:twoprops}

The proof of Theorem \ref{thm:maindecomp} consists of two main intermediate results, which are formulated in this section. The first step roughly tells us that in order to find a product measure close to $\nu$ in the $\W$ metric, it is enough to control the quantity $\tr \left (\HH(\nu) \right )$.
\begin{proposition} \label{prop:step1}
Let $\tilde \nu$ be a probability measure on $\DC$. Then there exists a product measure $\xi=\xi(\tilde \nu)$ such that
$$
\W(\tilde \nu, \xi) \leq \sqrt{n \tr \left ( \HH( \tilde \nu) \right )}.
$$
Moreover, one may take $\xi$ to be the unique product measure whose center of mass lies at the point $\int_{\DC} g_{\tilde \nu}(y) d \tilde \nu(y)$ which is equal to the center of mass of of $\tilde \nu$.
\end{proposition}

The second step, which is the more difficult one, is the following proposition which tells us that we can find a decomposition of $\nu$ via small tilts in a way that controls the matrix $\HH(\tilt_\theta \nu)$. For a matrix $A$ whose decreasing rearrangement of diagonal entries is denoted by $(\alpha_i)_{1 \leq i \leq n}$, we denote $\tr_k(A) := \sum_{i= \lceil k \rceil }^n \alpha_i$.
\begin{proposition} \label{prop:step2}
Define $D = \MW \left ( \{ g_\nu(y): y \in \DC \} \right ) $. Let $\nu$ be a probability measure on $\DC$ and define $f = \log \frac{d \nu}{d \mu}$. For all $\eps \in \left (0,  1/16 \right )$, there exists a measure $m$ supported on $B(0, \eps \sqrt{n}) \cap [-1,1]^n$, such that $\nu$ admits the decomposition
\begin{equation}\label{eq:decompprop}
\int_{\DC} \varphi d\nu = \int_{B(0, \eps \sqrt{n})} \left (\int_{\DC} \varphi d \tilt_\theta (\nu) \right ) d m(\theta)
\end{equation}
for every test function $\varphi: \DC \to \RR$ and which satisfies
\begin{equation}\label{eq:Hnusmall}
m \left (\theta: ~~ \tr_{k}(\HH(\tilt_\theta \nu)) \leq  \frac{2^8 \alpha D}{\eps} \right ) \geq 1 - \frac{1}{\alpha} - \frac{1}{n}, ~~~ \forall \alpha > 1,
\end{equation}
where $k \leq 2n e^{-1/(32 \eps^2)}$. 
Furthermore, under the additional assumption $\eps \leq \frac{1}{8 \sqrt{\log \frac{4n}{D} }}$, the above equation holds true with $k=1$.
\end{proposition}

We will also need the following easy lemma.
\begin{lemma} \label{lem:easy}
For every $\theta \in B(0, \eps \sqrt{n})$ and for all $y \in \DC$ one has
\begin{equation}\label{eq:lemeasy}
\left | \log \frac{d \tilt_\theta \nu}{d \nu}(y) \right | \leq 2 \eps n.
\end{equation}
\end{lemma}
\begin{proof}
We have by definition
$$
\log \frac{d \tilt_\theta \nu}{d \nu} = \langle y, \theta \rangle - \log \int \exp(\langle z, \theta \rangle ) d \nu(x).
$$		
Now, since $\theta \in B(0, \eps \sqrt{n})$ we have for all $y \in \DC$ that $|\langle y, \theta \rangle| \leq \eps n$. Consequently we also have $\left |\log \int e^{\langle \theta, y \rangle} d \nu \right | \leq \eps n$. The lemma follows.
\end{proof}

Given the above, the proofs of Theorem \ref{thm:maindecomp} and Theorem \ref{thm:main} follow easily. 

\begin{proof}[Proof of Theorem \ref{thm:maindecomp}]
Given the measure $\nu$, apply Proposition \ref{prop:step2} with the parameters $\alpha, \eps$ to find a measure $m$ on $B(0,\eps \sqrt{n}) \cap [-1,1]^n$ such that the decomposition \eqref{eq:decompprop} holds. Define $\Theta$ to be the set of $\theta \in \RR^n$ such that the event in equation \eqref{eq:Hnusmall} is satisfied with $k=1$. Then for each $\theta \in \Theta$, using Proposition \ref{prop:step1} with $\tilde \nu = \tilt_\theta \nu$, one concludes that
\begin{align*}
\W(\tilt_\theta \nu, \xi(\tilt_\theta \nu)) ~& \leq \sqrt{n \tr \left ( \HH(\tilt_\theta \nu) \right )} \\
& \leq \sqrt{\frac{ 2^8 \alpha n \MW( \{ g_\nu(y): y \in \DC \} ) }{\eps}} \\
& \stackrel{\eqref{eq:SF}}{\leq} \sqrt{\frac{2^8 \alpha n \Comp(\nu) }{\eps}}
\end{align*}
where, in the second inequality we applied Lemma \ref{lem:SF}. Finally, by Lemma \ref{lem:easy}, we have
\begin{align*}
\left | \int_{\RR^n} \KL(\tilt_\theta \nu \Vert \mu)  d m(\theta) - \KL(\nu \Vert \mu) \right | ~& = \left |\int_{\RR^n} \int_{\DC} \left ( \log \frac{d \tilt_\theta \nu}{d \nu} (y) +\log \frac{d \nu}{d \mu} (y) \right ) d \tilt_\theta \nu(y) d m(\theta) \right . \\ &~~~  \left . - \int_{\DC} \log \frac{d \nu}{d \mu} (y) d \nu(y) \right | \\
& \stackrel{\eqref{eq:decompprop}}{=} \left |\int_{\RR^n} \int_{\DC} \log \frac{d \tilt_\theta \nu}{d \nu} (y) d \tilt_\theta \nu(y) d m(\theta)  \right |\\ 
 & \stackrel{ \eqref{eq:lemeasy}}{\leq} 2 \int_{\RR^n} \int_{\DC} \eps n d \tilt_\theta \nu(y) d m(\theta) \leq 2 \eps n.
\end{align*}
\end{proof}

\begin{proof}[Proof of Theorem \ref{thm:main}]
Given the measure $\nu$, apply Proposition \ref{prop:step2} with the measure $\nu$ the parameter $\eps$ and with the choice $\alpha=3$. A consequence of formula \eqref{eq:Hnusmall} is the existence of $\theta \in B(0, \eps \sqrt{n})$ and of a diagonal projection matrix $\sigma$ of rank $n-k$, with $k \leq 2 n e^{-1/(32 \eps^2)}$, such that
$$
\tr(\sigma \HH(\tilt_\theta \nu)) \leq \frac{2^8 \alpha \Comp(\nu)}{\eps}.
$$
Define $\tilde \nu = \tilt_\theta \nu$ and $\tilde f = \log \frac{d \tilt_\theta \nu}{d \mu}$ so that
$$
\tilde f(y) = f(y) + \langle \theta, y \rangle - \log \int e^{\langle \theta, y \rangle} d \nu.
$$
Now, since $\theta \in B(0, \eps \sqrt{n})$ we have for all $y \in \DC$ that $|\langle y, \theta \rangle| \leq \eps n$. Consequently we also have $\left |\log \int e^{\langle \theta, y \rangle} d \nu \right | \leq \eps n$. We therefore get $|f - \tilde f| \leq 2 \eps n$. \\

Define $b = \int y d \tilde \nu(y)$, the center of mass of $\tilde \nu$. Let $\xi$ be the unique product measure whose center of mass is at $b$. Since, by definition, we have $|\langle g_\nu(y), e_i \rangle| \leq 1$ for all $y \in \DC$ and for all $i \in [n]$, it follows that $\langle e_i, \HH(\tilde \nu) e_i \rangle \leq 1$. According to Proposition \ref{prop:step1}, we therefore have
$$
\W(\tilde \nu, \xi) \leq \sqrt{n \tr(\HH(\tilde \nu))} \leq \sqrt{n \left (k+ \tr(\sigma \HH(\tilde \nu) ) \right ) }  
$$
which completes the first part of the theorem.

Roughly speaking, the second part of the theorem follows by considering the foliation with respect to $\sigma$ and invoking Proposition \ref{prop:step1} on each sub-cube separately. By rearranging the coordinates, we may assume without loss of generality that the diagonal entries of the matrix $\sigma$ are increasing. For each $y \in \mathcal{C}_k$, consider the sub-cube $A_y = \left \{x \in \DC; x_i=y_i, ~ 1 \leq i \leq k \right \}$ and let $\tilde \nu_y, \xi_y$ be the restrictions of the measures $\tilde \nu$ and $\xi$ to $A_y$, respectively, normalized to be probability measures. Let $b(y)$ be the center of mass of $\tilde \nu_y$ and recall that $b$ is the center of mass of $\tilde \nu$. Remark that for all $i>k$, for all $y \in \mathcal{C}_k$ and for all $x \in A_y$, one has that $\langle g_{\tilde \nu} (x), e_i \rangle = \langle g_{\tilde \nu_y}(x), e_i \rangle$. Using formula \eqref{eq:gcom}, we therefore get
$$
\langle b(y), e_i \rangle = \left \langle \int g_{\tilde \nu}(x) d \tilde \nu_y(x), e_i \right \rangle.
$$
Consequently, by the law of total variance, we have for all $i > k$,
$$
\mathrm{Var}_{X \sim \tilde \nu} \left [\langle g_{\tilde \nu}(X), e_i \rangle \right ] = \sum_{y \in \mathcal{C}_k} \tilde \nu(A_y) \left ( \mathrm{Var}_{Y \sim \tilde \nu_y} \left [ \langle g_{\tilde \nu}(Y), e_i \rangle \right ] + \langle b(y) - b, e_i \rangle^2 \right ).
$$
which gives, by definition of the matrix $\HH(\cdot)$,
\begin{equation}\label{eq:totalvariance}
\tr \left ( \sigma \HH(\tilde \nu) \right ) = \sum_{y \in \mathcal{C}_k} \tilde \nu(A_y) \left (\tr \left (\HH(\tilde \nu_y) \right ) + |\sigma(b(y) - b)|^2 \right ).
\end{equation}
For each $y \in \mathcal{C}_k$, we invoke Proposition \ref{prop:step1} on $\tilde \nu_y$ to conclude that the product measure $\tilde \xi_y$, whose center of mass lies at $b(y)$, satisfies
\begin{equation}\label{eq:distfiber}
\W(\tilde \xi_y, \tilde \nu_y) \leq \sqrt{n \tr \left (\HH(\tilde \nu_y) \right )}.
\end{equation}
Define $\xi' := \pi_I(\xi) \times \pi_{[n] \setminus I} (\tilde \nu)$ where $I = \{i; \sigma_{i,i} = 1\}$.
Since for all $y \in \mathcal{C}_k$ one has $\xi'(A_y) = \tilde \nu(A_y)$, we have
\begin{align*}
\W(\xi', \tilde \nu) ~& \leq \sum_{y \in \mathcal{C}_k} \tilde \nu(A_y) \W(\xi_y, \tilde \nu_y)  \\
& \leq \sum_{y \in \mathcal{C}_k} \tilde \nu(A_y) \left (\W(\xi_y,  \tilde \xi_y) + \W(\tilde \xi_y, \tilde \nu_y) \right ) \\
& \stackrel{\eqref{eq:distfiber}}{\leq} \sum_{y \in \mathcal{C}_k} \tilde \nu(A_y) \left ( \Vert \sigma (b - b(y)) \Vert_1 + \sqrt{n \tr \left (\HH(\tilde \nu_y) \right )} \right ) \\
& \leq \sqrt{n} \sum_{y \in \mathcal{C}_k} \tilde \nu(A_y) \left ( | \sigma (b - b(y)) | + \sqrt{\tr \left (\HH(\tilde \nu_y) \right )} \right ) \\
& \leq \sqrt{2 n \sum_{y \in \mathcal{C}_k} \tilde \nu(A_y) \left ( | \sigma (b - b(y)) |^2 + \tr \left (\HH(\tilde \nu_y) \right ) \right ) } \\
& \stackrel{ \eqref{eq:totalvariance} }{=} \sqrt{2 n \tr(\sigma \HH(\tilde \nu))} \leq 2^6 \sqrt{n \frac{\Comp(\nu)}{\eps}},
\end{align*}
which is the desired bound.
\end{proof}
\subsection{The first step: obtaining an estimate on $\W$ using $\HH$}
The proof of the first step is rather straightforward. We choose to prove it directly, but in fact it follows from a combination of two well-known inequalities: the Log-Sobolev inequality and the transportation-entropy inequality. The sketch of this direction goes as follows: let $\xi$ be the product measure having the same expectation as $\nu$. Then, a straightforward calculation gives that $\tr \HH(\nu) \approx \I_\xi(\nu)$ where $\I_\xi(\nu)$ is the Fisher information of $\nu$ with respect to $\xi$. A combination of the two above inequalities then gives $\W(\nu, \xi)^2 \leq 2 n \KL(\nu \Vert \xi) \leq n \I_\xi(\nu)$ up to constants.

We now give a direct proof of the proposition.
\begin{proof}[Proof of Proposition \ref{prop:step1}]
Let $U_1,\dots,U_n$ be independent random variables uniformly distributed in $[-1,1]$. Define $\overline g = (\overline g^1, \dots, \overline g^n) = \int g_{\nu}(y) d \nu(y)$ for all $i \in [n]$ (where $g_\nu$ is defined in equation \eqref{eq:defgnu} above). Let $Y=(Y^1,\dots,Y^n)$ be a random point in $\DC$ defined by
$$
Y^i = \begin{cases}
+1 & U_i \leq \overline g^i \\
-1 & U_i > \overline g^i.
\end{cases}
$$
We define $\xi$ to be the law of $Y$. Clearly, $\xi$ is a product measure. Let us now define a suitable coupling of $Y$ with a random variable $Z=(Z^1,\dots,Z^n)$ whose law is $\nu$. Consider the filtration $\FFF_i = \sigma(U_1, \dots, U_{i-1})$. Set also 
$$
J(j) := \{(y_1,\dots,y_n) \in \DC| ~~  y_k = Z^k, ~~ \forall 1 \leq k \leq j-1 \}.
$$
Define a vector $\Lambda(j) = (\Lambda^1(j),\dots,\Lambda^n(j))$ by the formula
$$
\Lambda(j) = \frac{\sum_{y \in J(j)} e^{f(y)} g(y)}{\sum_{y \in J(j)} e^{f(y)}}
$$
where we abbreviate $g(y) = g_\nu(y)$. Finally, we can define inductively
$$
Z^{i} = \begin{cases}
+1 & U_i \leq \Lambda^i(i) \\
-1 & U_i > \Lambda^i(i).
\end{cases}
$$
Remark that $\Lambda(j)$ is $\FFF_j$ measurable. Let us now show that $Z \sim \nu$. Define $Z(i)=(Z^1,\dots,Z^{i},0,\dots,0)$. First note that whenever $i \geq j$, one has
\begin{align*}
\Lambda^i(j) ~& = \frac{\sum_{y \in J(j), \atop y_i=-1} g(y) \left (e^{f(y)} + e^{f(y+2e_i)} \right ) }{\sum_{y \in J(j)} e^{f(y)}} \\
& \stackrel {\eqref{eq:defgp} }{=} 2 \frac{\sum_{y \in J(j), \atop y_i=-1} \partial_i e^{f}(y) }{\sum_{y \in J(j)} e^{f(y)}} \\
& = \frac{\sum_{y \in J(j)} \partial_i e^{f}(y) }{\sum_{y \in J(j)} e^{f(y)}} 
 \stackrel {\eqref{eq:nablaharmonic} }{=} \frac{\partial_i h(Z(j-1))}{h(Z(j-1))}.
\end{align*}
The last equation and the definition of $Z^{i}$ teach us that 
$$
\PP(Z^i = s | Z(i-1)) = \frac{1 + s \frac{\partial_i h(Z(i-1))}{h(Z(i-1))}}{2} = \frac{h(Z^1,\dots,Z^{i-1},s,0,\dots,0)}{h(Z(i-1))}, ~~ s = \pm 1.
$$
So,
$$
\PP(Z=(z^1,...,z^n)) = \prod_{i \in [n]} \PP \bigl (Z^i = z^i \bigl | Z(i-1) = z(i-1) \bigr . \bigr) = \prod_{i \in [n]} \frac{h(z(i))}{h(z(i-1))} = h(z) = e^{f(z)}
$$
where we have defined $z(i) = (z^1,\dots,z^i,0,\dots,0)$. This establishes the fact that $Z \sim \nu$. Moreover, by the last formula it is easily seen that that $Z | \FFF_i$ has the law $\frac{e^{f(\cdot)} \mathbf{1}_{\cdot \in J(i)}} {\sum_{y \in J(i)} e^{f(y)}}$. Consequently, we have
$$
\Lambda(j) = \EE[g(Z) | \FFF_j] = \EE[\Lambda(n) | \FFF_j], ~~ \forall j \in [n]
$$
so that $\Lambda(j)$ is a martingale. By definition of $\Lambda$ and $\HH(\nu)$ one may verify that 
$$
\mathrm{Cov}(\Lambda(n)) = \mathrm{Cov}(g(Z)) = \HH(\nu).
$$
The fact that $\Lambda(j)$ is a martingale implies that
$$
\mathrm{Var}[\Lambda^i(i)] \leq \mathrm{Var}[\Lambda^i(n)] = \langle e_i, \HH(\nu) e_i \rangle
$$
and, moreover, we have for all $i$,
$$
\EE[\Lambda(i)] = \EE[\Lambda(n)] = \EE[g(Z)] = \overline g.
$$
The last two equalities and the definition of $Z^i$ and $Y^i$ give us that $\EE[Z^i]=\EE[Y^i]$ and that
\begin{align*}
\PP(Z^i \neq Y^i) ~& = \EE \left |\PP(U_i \geq \overline g^i) - \PP(U_i \geq \Lambda^i(i) | \FFF_{i}) \right | \\
& = \frac{1}{2} \EE \bigl [ |\overline g^i - \Lambda^i(i)| \bigr] \\
& \leq \sqrt{\mathrm{Var}[\Lambda^i(i)]} \leq \sqrt{\langle e_i, \HH(\nu) e_i \rangle}.
\end{align*}
Consequently,
$$
\W(\xi,\nu) \leq \EE \left [ \sum_{i=1}^n \mathbf{1}_{ \{ Y^i \neq Z^i \}}  \right ] \leq \sum_{i \in [n]} \sqrt{\langle e_i, \HH(\nu) e_i \rangle} \leq \sqrt{n \tr(\HH)}.
$$
Finally note that by definition, the center of mass of $\xi$ lies at the point $\overline g$ which is by the identity \eqref{eq:gcom} equal to the center of mass of $\nu$. The proof is complete.
\end{proof}

\subsection{The second step: finding product-like tilts}
The proof of Proposition \ref{prop:step2} is based on several stochastic constructions, introduced henceforth.

\subsubsection{Stochastic constructions} \label{sec:stochastic}
Let a probability measure $\nu$ on $\DC$ be fixed and define the functions $v(x)=v_\nu(x),h(x)=h_\nu(x)$ and $g(y)=g_\nu(y)$ as in Section \ref{sec:preliminaries}. Let $\sigma:[-1,1] \times [0,\infty) \to \RR$ be the function 
$$
\sigma(x,t) = \begin{cases}
\mathbf{1}_{x \in \left (-\tfrac 1 2, \tfrac 1 2 \right )} & 0 \leq t < 1 \\
\mathbf{1}_{x \in (-1,1)} & t \geq 1.
\end{cases}
$$ 
By slight abuse of notation, for $x = (x_1, \dots, x_n) \in \CC$ define 
$$
\sigma(x,t) = \left (\begin{matrix}
\sigma(x_1,t) & 0 & 0 & \cdots \\
0 & \sigma(x_2,t) & 0 & \cdots \\
0 & 0 & \sigma(x_3,t) & \cdots \\
\cdots 
\end{matrix} \right ).
$$
Let $B_t$ be a standard Brownian motion in $\RR^n$ adapted to a filtration $\FF_t$. Our central construction is the following: let $X_t$ be the solution of the stochastic differential equation
\begin{equation}\label{eq:mainsde}
X_0 = 0, ~~ d X_t = \sigma(X_t,t)^{1/2} d B_t + \sigma(X_t,t) v(X_t) dt
\end{equation}
where $v(X_t) = v_\nu(X_t)$ is defined in equation \eqref{eq:defvx2}. 
\begin{remark} 
The function $\sigma(x,t)$ is defined in a way that $X_t \in [-1/2,1/2]^n$ for $t \leq 1$, $X_t \in \CC$ for all $t$ and $\lim_{t \to \infty} X_t \in \DC$. The particular choice of function is not important as long as one gets this sort of behavior, and in fact, as long as $\eps$ is smaller than the order $1/\log n$, it will be enough to define $\sigma(x,t) = \mathbf{1}_{x < 1}$ instead.
\end{remark}
Define also
$$M_t = \log h(X_t)$$
and
$$
\sigma_t = \sigma(X_t,t), ~ v_t = v(X_t).
$$
We have, by a simple calculation using It\^o's formula and by Fact \ref{fact:extmartingale},
$$
d h(X_t) = \left \langle \nabla h(X_t), \sigma_t^{1/2} d B_t + \sigma_t v_t dt \right \rangle
$$
and
\begin{equation}\label{eq:dMt}
d M_t = \langle \sigma_t^{1/2} v_t, d B_t \rangle + \frac{1}{2} |\sigma_t^{1/2} v_t|^2 dt .
\end{equation}
Finally, we define 
$$
q(x) = \sum_{y \in \DC} g(y) w(x,y) e^{f(y)}, ~~ \forall x \in \CC
$$ 
the harmonic extension of $g(x) e^{f(x)}$ to $\CC$. Consider the process
\begin{equation}\label{eq:defgt}
g_t := \frac{q(X_t)}{h(X_t)}.
\end{equation}

Define $X_\infty = \lim_{t \to \infty} X_t$. By the martingale convergence theorem and by the fact that $\sigma(x) = 1$ for $x \in (-1,1), t>1$ and $\sigma(\pm 1) = 0$, we have that almost surely $X_\infty \in \DC$. The following fact is a direct consequence of Girsanov's formula.
\begin{fact}
For any $t>0$, the random variable $X_\infty$ conditioned on $\FF_t$ has the law $y \to \frac{w(X_t,y) e^{f(y)}}{h(X_t)}$. In other words, for every test function $\varphi: \DC \to \RR$, one has
\begin{equation}\label{eq:com}
\EE \left [ \varphi(X_\infty) | \FF_t \right ] = \frac{1}{h(X_t)} \sum_{y \in \DC} \varphi(y) w(X_t,y) e^{f(y)}.
\end{equation}
\end{fact}
\begin{proof}
Fix $t>0$. Suppose that $\{B_s\}_{t \geq s}$ is a Brownian motion when the underlying Wiener space is equipped with a measure $P$. Let $Q$ be a measure, defined on the same underlying Wiener space, by
$$
\frac{dQ}{dP} = \exp \left (  - \int_t^\infty \left (\langle \sigma_s^{1/2} v_s, d B_s \rangle + \frac{1}{2} \left |\sigma_s^{1/2} v_s \right |^2 dt \right ) \right ).
$$
Then, by Girsanov's formula we have that the process $s \to B_{t+s} - B_t + \int_t^{t+s} \sigma_r^{1/2} v_r dr$ is a Brownian motion under the measure $Q$ and according to formula \eqref{eq:mainsde}, the process $\{X_s \}_{s \geq t}$ is a martingale under that measure, whose diffusion matrix is diagonal. Consequently we have that the distribution of $X_\infty | \FF_t$ under $Q$ has the density $w(X_t, \cdot)$, since the latter represents the harmonic measure on $\DC$ with respect to any diagonal martingale started at $X_t$. Using equation \eqref{eq:dMt}, we have that
$$
\frac{dP}{dQ} = \exp \left ( \int_t^\infty d M_t \right ) = e^{M_\infty - M_t} = \frac{h(X_\infty)}{h(X_t)} = \frac{e^{f(X_\infty)}}{h(X_t)}.
$$
Thus, we have under the measure $P$ that
\begin{align*}
\EE_P[\varphi(X_\infty) | \FF_t] ~& = \left .\EE_Q \left [ \frac{dP}{dQ} \varphi(X_\infty)  \right | \FF_t\right ] \\
& = \left . \EE_Q \left [ \frac{e^f(X_\infty)}{h(X_t)} \varphi(X_\infty) \right | \FF_t \right ] \\
& = \frac{1}{h(X_t)}\sum_{y \in \DC} e^{f(y)} w(X_t, y) \varphi(y).
\end{align*}
The proof is complete.
\end{proof}
\begin{remark}
A different way to see that the above identity is correct without using Girsanov's inequality is simply to calculate the It\^o differential of the process $p_t := \frac{e^{f(y)} w(X_t,y) }{ h(X_t)}$ and observe that it is a martingale. Therefore, one has $\PP(X_\infty = y | \FF_t) = \EE[p_\infty | \FF_t] = p_t$.
\end{remark}
Define a mapping  $\eta: \mathrm{Int}(\CC) \to \RR^n$ by 
$$
\langle \eta(x), e_i \rangle = \log \sqrt{\frac{1 + \langle x, e_i \rangle}{1- \langle x, e_i \rangle}}.
$$
We have by definition, for every $x \in \CC$ and $y \in \DC$, $\exp(\langle \eta(x), y \rangle) = \prod_{i \in [n]} \frac{1 + x_i y_i}{\sqrt{(1+x_i)(1-x_i)}} = C_x w(x,y)$ with $C_x$ depending only on $x$. Therefore
\begin{equation}\label{eq:eta}
\int_{\DC} \varphi d (\tilt_{\eta(X_t)} \nu) = \frac{1}{h(x)} \sum_{y \in \DC} \varphi(y) w(X_t, y) e^{f(y)} \stackrel{\eqref{eq:com}}{=} \EE[\varphi(X_\infty) | \FF_t ].
\end{equation}
The following two corollaries follow immediately from equation \eqref{eq:com}.
\begin{corollary}
For every stopping time $\tau$ such that $X_\tau \in \mathrm{Int}(\CC)$ almost surely, one has the following decomposition of the measure $\nu$: for every test function $\varphi:\DC \to \RR$,
\begin{equation}\label{eq:decomposition}
\int_{\DC} \varphi d \nu = \EE \left [  \int_{\DC} \varphi d (\tilt_{\eta(X_\tau)} \nu) \right ].
\end{equation}
\end{corollary}
\begin{proof}
Since $\EE[\varphi(X_\infty) | \FF_t]$ is a martingale, we have by the optional stopping theorem
$$
\int_{\DC} \varphi d \nu = \EE[\varphi(X_\infty)] = \EE \bigl [ \EE[ \varphi(X_\infty) | \FF_{\tau} ] \bigr ] \stackrel{\eqref{eq:eta}}{=} \EE \int_{\DC} \varphi d (\tilt_{\eta(X_\tau)} \nu).
$$
\end{proof}

\begin{corollary}
One has the identities
$$
v_t = \EE[v(X_\infty)| ~ \FF_t], \mbox{ and } g_t = \EE[g(X_\infty)| ~ \FF_t].
$$
In particular, the processes $v_t$ and $g_t$ are martingales.
\end{corollary}
\begin{proof} 
Observe that, by definition,
$$
g_t = \frac{1}{h(X_t)} \sum_{y \in \DC} g(y) w(X_t,y) e^{f(y)}
$$	
and
$$
v_t = \frac{\nabla h(X_t)}{h(X_t)} \stackrel{\eqref{eq:nablaharmonic}}{=} \frac{1}{h(X_t)} \sum_{y \in \DC} v(y) w(X_t,y) e^{f(y)}.
$$
Apply equation \eqref{eq:com} with the choices $\varphi(y)=v(y)$ and $\varphi(y)=g(y)$.
\end{proof}

The following calculation is central to our proof. It is the consequence of a straightforward calculation using It\^{o}'s formula, and its proof is postponed to the appendix.
\begin{fact} \label{fact:dgt}
One has,
\begin{equation}\label{eq:dgt}
d g_t = \Gamma_t \sigma_t^{1/2} d B_t
\end{equation}
where
\begin{equation}\label{eq:defGammat}
\Gamma_t := \frac{\nabla q(X_t) }{ h(X_t)} - g_t \otimes v_t.
\end{equation}
\end{fact}
Our final definition is that of the matrix-valued stochastic process
\begin{equation}\label{eq:defHt}
H_t := \EE[ (g_\infty - g_t)^{\otimes 2} | \FF_t ].
\end{equation}
According to formula \eqref{eq:com} and since $g_t$ is a martingale, we have that also

	\begin{align*}
		H_t &= \frac{1}{h(X_t)} \sum_{y \in \DC} e^{f(y)} w(X_t,y) g(y)^{\otimes 2}  - g_t \otimes g_t \\
		& \stackrel{\eqref{eq:eta}}{=} \int_{\DC} g_{\nu}(y)^{\otimes 2} d \tilt_{\eta(X_t)} \nu(y)  - \left ( \int_{\DC} g_{\nu}(y) d \tilt_{\eta(X_t)} \nu(y) \right )^{\otimes 2}.
	\end{align*}
Recalling that
$$
\HH \bigl (\tilt_{\eta(X_t)} \nu \bigr) = \int_{\DC} g_{\tilt_{\eta(X_t)} \nu}(y)^{\otimes 2} d \tilt_{\eta(X_t)} \nu(y)  - \left ( \int_{\DC} g_{\tilt_{\eta(X_t)} \nu}(y) d \tilt_{\eta(X_t)} \nu(y) \right )^{\otimes 2},
$$
we immediately see that $H_0 = \HH(\nu)$. Furthermore, while the matrices $\HH(\tilt_{\eta(X_t)} (\nu))$ and $H_t$ are not exactly equal, they are close to each other when the coordinates of $X_t$ are small, as shown in the next technical lemma, whose proof appears in the appendix.

\begin{lemma} \label{lemHtHt}
Let $\theta \in \RR^n$ and let $\nu,\tilde \nu$ be probability measures on $\DC$. Define	
$$
A = \int_{\DC} g_{\nu}(y)^{\otimes 2} d \tilde \nu(y)  - \left ( \int_{\DC} g_{\nu}(y) d \tilde \nu(y) \right )^{\otimes 2}
$$	
and
$$
B = \int_{\DC} g_{\tilt_\theta \nu} (y)^{\otimes 2} d \tilde \nu(y)  - \left ( \int_{\DC} g_{\tilt_{\theta} \nu}(y) d \tilde \nu(y) \right )^{\otimes 2}.
$$
Then for all $1 \leq i \leq n$,
$$
e^{-4 \|\theta \|_\infty} B_{i,i} \leq A_{i,i} \leq e^{4 \|\theta \|_\infty} B_{i,i}.$$
\end{lemma}

An immediate consequence of this lemma is that under the event $X_t \in [-1/2,1/2]^n$, we have $\exp (2 |\langle \eta(X_t), e_i \rangle |) = \frac{1+ | \langle X_t, e_i \rangle |}{1 - | \langle X_t, e_i \rangle |} \leq 3$ and therefore
\begin{equation}\label{eq:HtHt}
X_t \in [-1/2,1/2]^n \Rightarrow \tr (\sigma \HH(\tilt_{\eta(X_t)} \nu ) ) \leq 9 \tr \left (\sigma H_t \right )
\end{equation}
for every positive-definite diagonal matrix $\sigma$.

Recall that our final objective is to control $\tr \HH(\tilt_{\eta(X_t)} (\nu))$. In view of the above equation it therefore suffices to control the trace of $H_t$. A key fact in the proof of Proposition \ref{prop:step2}, which we will see later on, is that the matrix $H_t$ is, in a sense, controlled by the matrix $\Gamma_t$. Intuitively, this means that the martingale $g_t$ has to be moving quickly whenever $\mathrm{Cov}(g_\infty)$ is big.

\subsubsection{Proof of Proposition \ref{prop:step2}}

The following simple observation will help us exploit the complexity condition.
\begin{observation} \label{obs:convhull}
	For every $t \geq 0$ we have, almost surely, $g_t \in \mathbf{Conv}(\{g(y): y \in \DC \} )$.
\end{observation}
\begin{proof}
	Since $g_t$ is a martingale, and since almost surely there exists some $y \in \DC$ such that $g_\infty = g(y)$, it follows that $g_t$ can be written as 
	$$
	g_t = \sum_{y \in \DC} g(y) \PP(X_\infty = y | \FF_t)
	$$
	which is a convex combination of vectors in the set $\{g(y), y \in \DC\}$.
\end{proof}

Next, we will need the following lemma which will help us make sense of the matrix $\Gamma_t$ defined in \eqref{eq:defGammat}. 
\begin{lemma} \label{lem:diagequal}
	One has almost surely, for all $t \geq 0$ and all $i \in [n]$,
	$$
	\left (\EE[g_\infty \otimes v_\infty | \FF_t ] \right )_{i,i} = \left (\frac{\nabla q (X_t) }{h(X_t)} \right )_{i,i}.
	$$
\end{lemma}

\begin{proof}
	Fix $i \in [n]$. For all $y \in \DC$ define $y_+$ to be the point equal to $y$ on every coordinate except maybe the $i$th coordinate, where it is equal to $+1$. Define $y_-$ analogously. We have for all $y \in \DC$,
	$$
	q_i(y) \stackrel{\eqref{eq:defgp} }{=}  e^{f(y)} \frac{e^{f(y_+)} - e^{f(y_-) } }{ e^{f(y_+)} + e^{f(y_-)}}
	$$
	and thus	
	$$
	\partial_i q_i(y) = \frac{1}{2} \frac{ \left (e^{f(y_+)} - e^{f(y_-) } \right )^2 }{ e^{f(y_+)} + e^{f(y_-)}}. 
	$$
	On the other hand, with the help of equation \eqref{eq:defgp} we have
	$$
	g(y)_i v(y)_i = \frac{e^{f(y_+)} - e^{f(y_-) } }{ e^{f(y_+)} + e^{f(y_-)}} \frac{e^{f(y_+)} - e^{f(y_-) }}{ 2 e^{f(y)} }.
	$$
	Consequently, we have that for all $y \in \DC$,
	\begin{equation} \label{eq:partial}
	\partial_i q_i(y) = e^{f(y)} g(y)_i v(y)_i
	\end{equation}
	
	Summing up, we get
	\begin{align*}
	\EE[(g_\infty \otimes v_\infty)_{i,i} | \FF_t ] ~& \stackrel{\eqref{eq:com}}{=} \frac{1}{h(X_t)} \sum_{y \in \DC} w(X_t,y) e^{f(y)} v(y)_i g(y)_i \\
	& \stackrel{\eqref{eq:partial}}{=} \frac{1}{h(X_t)} \sum_{y \in \DC} w(X_t,y) \partial_i q_i(y) \\
	& \stackrel{\eqref{eq:nablaharmonic}}{=} \frac{\partial_i q_i(X_t)}{h(X_t)}
	\end{align*}
	which completes the proof.	
\end{proof}

Now, consider the matrix
\begin{equation}
A_t := \EE[ (g_\infty - g_t) \otimes (v_\infty - v_t) | \FF_t ] = \EE[g_\infty \otimes v_\infty | \FF_t] - g_t \otimes v_t
\end{equation}
(where the second equality follows from the fact that $g_t$ and $v_t$ are martingales). The result of Lemma \ref{lem:diagequal} combined with equation \eqref{eq:defGammat} tells us that
\begin{equation}\label{eq:equaltraces}
\tr(\sigma_t^{1/2} \Gamma_t) = \tr(\sigma_t^{1/2} A_t).
\end{equation}
Our next objective is to prove:
\begin{lemma} \label{lem:HtAt}
We have almost surely for all $t \geq 0$,
\begin{equation}\label{eq:vtgt}
\tr(\sigma_t^{1/2} H_t) \leq 4 \tr( \sigma_t^{1/2} A_t  ).
\end{equation}
\end{lemma}

In order to prove this lemma, we first formulate an intermediate technical lemma, whose proof is postponed to the appendix.
\begin{lemma} \label{lem:vtcoord}
For all $t \geq 0$ and $i \in [n]$ we have almost surely	
\begin{equation}\label{eq:vtcoord}
\langle v_t, e_i \rangle = \left . \EE \left [ \frac{\langle e_i, g_\infty \rangle}{1 + \langle X_t, e_i \rangle \langle e_i, g_\infty \rangle} \right | \FF_t \right ] = \EE \left [ \zeta_{\langle X_t, e_i \rangle} \left (\langle e_i, g_\infty \rangle\right ) | \FF_t \right ].
\end{equation}
and
\begin{equation}\label{eq:vtcoord2}
\EE[ \langle g_\infty - g_t, e_i \rangle  \langle v_\infty, e_i \rangle | \FF_t ] = \EE[ \langle g_\infty - g_t, e_i \rangle \zeta_{\langle X_t, e_i \rangle}(\langle g_\infty, e_i \rangle) | \FF_t].
\end{equation}
Moreover, for all $x \in \CC$ and $i \in [n]$ one has
\begin{equation}\label{eq:vtbound2}
\langle x, e_i \rangle \langle v(x), e_i \rangle \leq 1
\end{equation}
and under the additional assumption $x \in \left [-\tfrac{1}{2}, \tfrac{1}{2} \right ]^n$,
\begin{equation}\label{eq:vtbound}
|\langle v(x), e_i \rangle| \leq 2.
\end{equation}
\end{lemma}

\begin{proof}[Proof of Lemma \ref{lem:HtAt}]
Fix a coordinate $i \in [n]$ and define $G=\langle g_\infty, e_i \rangle, V = \langle v_\infty, e_i \rangle$, $x = \langle X_t, e_i \rangle$ and $\overline G =\langle g_t, e_i \rangle$. Since $g_t$ is a martingale we have $\overline G = \EE [G | \FF_t]$. Our main step will be to show that
\begin{equation}\label{eq:lemgoal1}
 \EE[ (G-\overline G) V | \FF_t ] \geq \frac{1}{4} \EE[ (G-\overline G)^2 | \FF_t ].
\end{equation}
Assuming the latter inequality, defining $\sigma_i = (\sigma_t^{1/2})_{i,i}$ we can write
\begin{align*}
\tr(\sigma_t^{1/2} A_t)  ~& = \EE \left . \left [ \sum_{i=1}^n \sigma_i \langle g_\infty - g_t, e_i \rangle \langle v_\infty - v_t, e_i \rangle \right  | \FF_t \right ] \\
~& = \EE \left . \left [ \sum_{i=1}^n \sigma_i \langle g_\infty - g_t, e_i \rangle \langle v_\infty, e_i \rangle \right  | \FF_t \right ] \\
& \stackrel{\eqref{eq:lemgoal1}}{\geq} \frac{1}{4} \EE\left . \left [ \sum_{i=1}^n \sigma_i \langle g_\infty - g_t, e_i \rangle^2 \right  | \FF_t \right ]  = \frac{1}{4} \tr(\sigma_t^{1/2} H_t), 
\end{align*}
which finishes the proof. In order to prove \eqref{eq:lemgoal1}, we first note that the identity
\begin{equation}\label{eq:GV}
\EE[ (G-\overline G)V | \FF_t ] = \EE[ (G-\overline G) \zeta_{x}(G) | \FF_t]
\end{equation}
follows from \eqref{eq:vtcoord2}. A calculation shows that for all $x,y \in (-1,1)$,
$$
\frac{d}{dy} \zeta_{x}(y) = \frac{d}{dy} \frac{y}{1+xy} = \frac{1}{(1+xy)^2} \geq \frac{1}{4}
$$
which implies that
\begin{equation}\label{eq:dz}
\left | \zeta_{x}(G) - \zeta_x(\overline G) \right | \geq \frac{1}{4} |G - \overline G|.
\end{equation}
Therefore,
\begin{align*}
\EE[ (G-\overline G)V | \FF_t ] ~& \stackrel{\eqref{eq:GV}}{=} \EE[ (G-\overline G) \zeta_{x}(G) | \FF_t] \\
&= \EE \left [ (G-\overline G) (\zeta_{x}(G) - \zeta_x(\overline G)) | \FF_t \right ] \\
&= \EE \left [ \left |G-\overline G\right | \cdot \left |\zeta_{x}(G) - \zeta_x(\overline G) \right | | \FF_t \right ] \\
&\stackrel{ \eqref{eq:dz} }{\geq} \frac{1}{4} \EE \left [ \left (G-\overline G\right )^2  | \FF_t \right ],
\end{align*}
which is \eqref{eq:lemgoal1}. The proof is complete.
\end{proof}
Define for all $t$,
$$
I_t = \left | \left \{i \in [n]; ~ \langle e_i, \sigma_t e_i \rangle = 0 \right \} \right |.
$$
and for all $0 < \eps < 1/2$, define the stopping time
$$
\TTT_{\eps} = \min \left \{ t>0: ~ \Vert X_t \Vert_2 = \eps \sqrt{n} \mbox{ or } I_t \geq 2 \exp \left (- \frac{1}{32 \eps^2} \right ) n \right \} \wedge 1.
$$
Observe that by definition $\langle e_i, H_t e_i \rangle = \mathrm{Var}[\langle g_\infty, e_i \rangle | \FF_t] \leq 1$. Therefore, for all $t \leq \TTT_\eps$, 
\begin{align} 
\tr(\sigma_t^{1/2} H_t) \geq \tr(H_t) - I_t \geq \tr(H_t) - 2 \exp(-1/(32 \eps^2)) n. \label{eq:traces}
\end{align}

The following lemma is the main point where Gaussian width comes to play. Its proof is, roughly speaking, an application of the divergence theorem for the vector field $\frac{q}{h}$.
\begin{lemma} \label{lem:divergence}
Let $B_t$ be a Brownian motion and let $g_t$ be a martingale, both adapted to a filtration $\FF_t$. Suppose that $d g_t = \tilde \Gamma_t d B_t$ for some matrix-valued process $\tilde \Gamma_t$ satisfying $\tr \left (\tilde \Gamma_t \right ) \geq 0$ almost surely for all $t$. Assume that there exists a set $K \subset \RR^n$ such that for all $t$ one has $g_t \in K$. Then one has for all $t>0$ and $\alpha > 1$,
$$
\PP \left ( \min_{s \leq t} \tr(\tilde \Gamma_s) > \alpha \tfrac{\MW(K)}{\sqrt{t}} \right ) < \frac{1}{\alpha}.
$$
\end{lemma}
\begin{proof}
The key idea is to observe that, by an application of It\^{o}'s isometry for the processes $g_{t}$ and $B_{t}$, we have that
\begin{align*}
\EE [ \langle g_{t}, B_{t} \rangle] ~& = \EE \left [ \int_0^{t} \tr \left (\tilde \Gamma_s \right ) ds \right ] \\
& \geq \EE \left [ t \min_{0 \leq s\leq t} \tr(\tilde \Gamma_s) \right ].
\end{align*}
On the other hand, since $g_t \in K$ almost surely for all $t$, we have
$$
\MW(K) \geq \EE \left [ \left \langle g_{t}, \tfrac{B_t}{\sqrt{t}} \right \rangle \right ].
$$
Therefore,
$$
\EE \left [ \min_{s\leq t} \tr(\tilde \Gamma_s) \right ] \leq \frac{\MW(K)}{\sqrt{t}}.
$$
Since, by assumption, $\tr(\tilde \Gamma_t)$ is nonnegative, the lemma is concluded via an application of Markov's inequality.
\end{proof}

The following statement follows easily from Gaussian tail bounds.
\begin{fact} \label{fact:tailmax}
Let $Z_t$ be a (one dimensional) It\^{o} process satisfying
$$
Z_0 = 0; ~~ d Z_t = r_t d B_t
$$
with $r_t$ being an adapted process satisfying $|r_t| < r$ almost surely, for some $r>0$. Then,
$$
\PP \left ( \max_{s \leq t} |Z_s| \geq \alpha \right ) < 2 \exp \left (-\frac{\alpha^2}{8r^2 t} \right )
$$
for all $\alpha,t>0$.
\end{fact}

As a corollary, we get
\begin{lemma} \label{lem:Tbig}
For all $n \geq 4$ and for all $\eps \in \left (0 ,1/8 \right )$, we have
$$
\PP \left (\TTT_\eps \leq \frac{\eps^2}{16} \right ) < 1/n.
$$
\end{lemma}
\begin{proof}
We have by It\^{o}'s formula and by \eqref{eq:mainsde},
$$
d \| X_t \|_2^2 = 2 \langle X_t, d X_t \rangle + d [X]_t = 2 \langle X_t, \sigma_t^{1/2} d B_t \rangle + 2 \langle X_t, \sigma_t v_t \rangle dt + \tr(\sigma_t) dt.
$$
The bound \eqref{eq:vtbound2} implies that $\langle X_t, e_i \rangle \langle v_t, e_i \rangle \leq 1$ for all $i \in [n]$ which, together with that fact that $\sigma_t$ is dominated by the identity, gives
$$
d \| X_t \|_2^2 \leq 2 \langle X_t, \sigma_t^{1/2} d B_t \rangle + 3n dt.
$$
Next, Fact \ref{fact:tailmax} with the fact that $\|X_t\|_2 \leq \eps \sqrt{n}$ for all $t \leq \TTT_\eps$, implies that
$$
\PP \left ( \max_{s \leq t \wedge \TTT_\eps} \int_0^s \langle X_r, \sigma_r^{1/2} d B_r \rangle \geq n \eps^2 / 2 \right ) < 2 \exp \left (-n \eps^2 / (8t) \right ).
$$
Setting $t = \frac{\eps^2}{16}$ gives
$$
\PP \left ( \max_{s \leq \tfrac{\eps^2}{16} \wedge \TTT_\eps } \| X_s\|_2^2 \geq \eps^2 n \right ) \leq 2 \exp(-n).
$$
Next, define 
$$
J_t = \left | \left \{ i \in [n]; ~ \max_{s \leq t} |\langle B_s, e_i \rangle| \geq 1/4 \right \} \right |.
$$ 
Then by Fact \ref{fact:tailmax}, $J_t$ is the sum of $n$ independent Bernoulli random variables, each with expectation bounded by $2\exp(-1/(8\eps^2))$. Recall that, by the Chernoff-Hoeffding Theorem, if $X_i$ are independent Bernoulli random variables with expectation $p$ and $S = \tfrac{1}{n} \sum_{i=1}^n X_i$, then
$$
\PP( S > \alpha ) \leq \left (\frac{p}{\alpha}\right )^{ \alpha n } \left (\frac{1-p}{1-\alpha}\right )^{ (1-\alpha) n } \leq \left (\frac{p}{\alpha}\right )^{ \alpha n } e^{\tfrac{3}{2} \alpha (1-\alpha) n} \leq \left (\frac{e^{3/2} p}{\alpha}\right )^{\alpha n }  , ~~ \forall p < \alpha < \tfrac{1}{2}.
$$
where the second inequality uses the fact that $\frac{1}{1-\alpha} \leq e^{3\alpha/2}$ for $\alpha \leq \tfrac{1}{2}$.
Taking $p=2\exp(-1/(8\eps^2))$ and $\alpha = \max(p^{1/4}, 1/n)$, this gives
\begin{align*}
\PP(J_t \geq 2 \exp(-1 / (32 \eps^2)) n) ~& = \PP \left (J_t \geq \max\left  (2 \exp(-1 / (32 \eps^2)) n,1 \right ) \right ) \\
& \leq \exp \left (  \tfrac{3}{4} \max(n p^{1/4}, 1)  \left ( \log p + 2 \right )\right ).
\end{align*}
The assumption $\eps < 1/8$ amounts to $p < e^{-7}$ and the expression $p^{1/4} (\log p + 2)$ is decreasing on the interval $(0,e^{-7})$. This implies $n p^{1/4} (\log p + 2) \leq - 4 \log n + 2$ for $p \geq n^{-4}$. We therefore get $\frac{3}{4} \max(n p^{1/4}, 1)\left (\log p + 2 \right ) \leq - 3 \log n + 2$ for all $p < e^{-7}$. Consequently,
$$
\PP(J_t > 2 \exp(-1 / (32 \eps^2)) n) \leq \frac{8}{n^3} \leq \frac{1}{2n}.
$$
By \eqref{eq:vtbound} one has $|\langle v_t, e_i \rangle| \leq 2$ for all $i \in [n]$ and $t \in [0,1]$ which gives
$$
\max_{s \leq t} |\langle B_s, e_i \rangle| < 1/4 \Rightarrow \max_{s \leq t} |\langle X_s, e_i \rangle| < 1/2, ~ \forall t \leq 1/8.
$$
Combining this with the fact that $I_t$ is increasing for $t \in [0,1]$, we conclude that $I_t \leq J_t$ for all $t \leq 1/8$. Finally, remarking that
$$
\TTT_\eps < t \leq 1 \Rightarrow  \max_{s \leq t \wedge \TTT_\eps } \| X_s\|_2^2 \geq \eps^2 n \mbox{ or } J_t \geq 2 \exp(-1 / (32 \eps^2)) n,
$$
and applying a union bound with respect to the events in the last display completes the proof.
\end{proof}

We can finally prove:
\begin{proof}[Proof of Proposition \ref{prop:step2}]
We intend to invoke Lemma \ref{lem:divergence} with the processes $g_t$ and $\tilde \Gamma_t = \sigma_t^{1/2} \Gamma_t$ where $\Gamma_t$ is defined as in \eqref{eq:defGammat}. Equations \eqref{eq:equaltraces} and \eqref{eq:vtgt}, together with the fact that $H_t$ is positive-definite, give
\begin{equation}\label{eq:traces2}
4 \tr \left (\tilde \Gamma_t \right ) \geq \tr \left (\sigma_t^{1/2} H_t \right ) \geq 0.
\end{equation}
Observation \ref{obs:convhull} ensures that $g_t \in K := \mathrm{Conv} \left ( \{g(y): y \in \DC \} \right ) $. Therefore we may invoke Lemma \ref{lem:divergence} with the choice $t = \frac{\eps^2}{16}$ to get
$$
\PP \left ( \min_{s \leq \tfrac{\eps^2}{16}} \tr \left (\sigma_s^{1/2} H_s \right ) > \tfrac{16 \alpha \MW(K)}{\eps} \right ) \leq \PP \left ( \min_{s \leq \tfrac{\eps^2}{16}} \tr(\tilde \Gamma_s) > \tfrac{4 \alpha \MW(K)}{\eps} \right ) < \frac{1}{\alpha}.
$$
Define
$$
\tau := \min \left  \{t : \tr \left (\sigma_t^{1/2} H_t \right ) \leq \tfrac{16 \alpha \MW(K) }{\eps} \right \} \wedge \TTT_\eps.
$$
A union bound gives
$$
\PP \left ( \tau < \TTT_\eps \right ) \geq 1 - \frac{1}{\alpha} - \PP \left (\TTT_\eps \leq \frac{\eps^2}{16} \right ).
$$
Remark that if either $D > 2^{-4} n$ or $n < 4$ then the result of the proposition follows trivially by taking $m$ to be supported at $0$, thus we may assume $n \geq 4$ and $\eps < 1/8$. Using Lemma \ref{lem:Tbig} therefore gives
\begin{align}\label{eq:taubound1}
\PP \left ( \tau < \TTT_{\eps} \right ) ~& \geq 1 - \frac{1}{\alpha} - \frac{1}{n}.
\end{align}
Applying equation \eqref{eq:decomposition} with the stopping time $\tau$ which tells us that for every test function $\varphi$,
$$
\int \varphi d \nu = \EE \left [ \int \varphi d \left ( \tilt_{\eta(X_\tau)} (\nu) \right ) \right ].
$$
Set the measure $m$ to be the law of $\eta(X_\tau)$, so that equation \eqref{eq:decompprop} holds. Since $\TTT_\eps \leq 1$ and by the definition of $\sigma_t$, we have $\Vert X_\tau \Vert_\infty \leq 1/2$ and $\Vert X_\tau \Vert_2 \leq \eps \sqrt{n}$. Since $|\eta'(z)| \leq 4/3$ for all $|z| < 1/2$ this implies that for $\theta : = \eta(X_\tau)$ one has $\Vert \theta \Vert_\infty \leq 2/3$ and $\Vert \theta \Vert_2 \leq \eps \sqrt{n}$. Therefore, the measure $m$ is supported on $[-1,1]^n \cap B(0, \eps \sqrt{n})$. 

Next, note that by definition of $\TTT_\eps$, we have
\begin{equation}\label{eq:sigmasmall}
\tr(\sigma_\tau^{1/2}) \geq n - 2 n e^{-1/(32 \eps^2)}.
\end{equation}
Apply equation \eqref{eq:HtHt} and use the definition of $\tau$ to get
$$
\tau < \TTT_\eps \Rightarrow \tr \left ( \sigma_\tau^{1/2} \HH(\tilt_{\eta(X_\tau)}(\nu)) \right ) \leq 10 \tr \left (\sigma_\tau^{1/2} H_\tau \right ) \leq \tfrac{2^8 \alpha \MW(K) }{\eps}
$$
which, together with equations \eqref{eq:taubound1} and \eqref{eq:sigmasmall} implies Equation \eqref{eq:Hnusmall}.  \\

It remains to show that under the additional assumption $\eps \leq \tfrac 1 8 \left (\log(4n/\MW(K) ) \right  )^{-1/2} \leq 1$, Equation \eqref{eq:Hnusmall} holds with $k=1$. To that end, an application of equation \eqref{eq:traces} gives
$$
\tau < \TTT_\eps \Rightarrow \tr(H_{\tau}) \leq \tfrac{16 \alpha \MW(K)}{\eps} + 2 \exp(-1 / (32 \eps^2)) n \leq \tfrac{20 \alpha \MW(K)}{\eps}.
$$
Applying Equation \eqref{eq:HtHt}, in a similar manner to the above, completes the proof.
\end{proof}

\section{The large deviation framework}

In this section we prove Corollary \ref{cor:main} and Theorem \ref{thm:largedev}. The following is a trivial consequence of the chain rule. Its proof is postponed to the appendix.
\begin{fact} \label{fact:chain}
	Let $\nu$ be measure on $\DC$ and let $\xi$ be a product measure satisfying $\EE[\xi] = \EE[\nu]$. Then $\KL(\nu \Vert \mu) \geq \KL(\xi \Vert \mu)$.
\end{fact}

\begin{proof} [Proof of Corollary \ref{cor:main}]
	Define $\hat f(x) = f(x) - \log \int e^{f} d \mu$, so that the measure $\nu$ defined by $d \nu = e^{\hat f(x)} d \mu$ is a probability measure. 
	Fix $\eps \in (0, \tfrac{1}{16})$ whose value will be chosen later on. Apply Theorem \ref{thm:main} with $\nu, \eps$ to obtain (using the second part of the theorem) a function $\tilde f: \DC \to \RR$ and the measures $\tilde \nu$ and $\xi'$ and a subset $I \subset [n]$. The theorem ensures us that
	\begin{equation}\label{eq:ens1}
	\max_{y \in \DC} | \hat f(y)-\tilde f(y)| \leq \eps n
	\end{equation}
	and that
	$$
	\W(\tilde \nu, \xi') \leq 2^6 \sqrt{\frac{n \Comp(f)}{\eps}}.
	$$				
	The last inequality clearly implies
	\begin{equation}\label{eq:xitnu}
	\left |\int f d \xi' - \int f d \tilde \nu \right | \leq 2^6 \Lip(f) \sqrt{\frac{n \Comp(f)}{\eps}}.
	\end{equation}
	Moreover, since $\xi' = \pi_I(\xi') \times \pi_{[n] \setminus I}(\tilde \nu)$,
	$$
	\int f (y) d \xi'(y) = \int_{\{-1,1\}^{[n] \setminus I} } \left (\int_{\{-1,1\}^{I}} f (z,w) d \pi_I (\xi' ) \bigl (z \bigr) \right ) d \pi_{[n] \setminus I}  ( \tilde \nu ) \bigl (w \bigr ).
	$$ 
	The above identity implies the existence of $w_0 \in \{-1,1\}^{ [n] \setminus I }$ for which $\int_{\{-1,1\}^I} f(z,w_0) d \pi_I (\xi') (z) \geq \int f d \xi'$.
	Define $\xi$ to be the restriction of $\xi'$ to $\{w_0\} \times \{-1,1\}^{I}$ and observe that $\xi$ is a product measure, since the marginal of $\xi'$ on the subset $I$ is a product measure. The above amounts to
	$$
	\int f d \xi \geq \int f d \xi' \stackrel{\eqref{eq:xitnu}}{\geq} \int f d \tilde \nu - 2^6 \Lip(f) \sqrt{\frac{n \Comp(f)}{\eps}}.
	$$
	Moreover, equation \eqref{eq:ens1} implies that
	$$
	\left |\int \tilde f d \tilde \nu - \left (\int f d \tilde \nu - \log \int e^{f} d \mu \right ) \right | \leq \eps n.
	$$
	By Fact \ref{fact:chain}, together with the chain rule for relative entropy, we have that
	\begin{align*}
	\KL(\tilde \nu \Vert \mu) ~& \geq \KL \bigl (\pi_I(\tilde \nu) \times \pi_{[n] \setminus I}(\tilde \nu) \Vert \mu \bigr)  \\
	& \geq \KL(\xi' \Vert \mu) \\
	& \geq \KL(\xi \Vert \mu) - (n-|I|) \log 2.
	\end{align*}	
	Since $e^{\tilde f}$ is a probability density, and since $n-|I| \leq 2n e^{-1/(32 \eps^2)}$, we therefore get
	$$
	\int \tilde f d \tilde \nu = \int \tilde f e^{\tilde f} d \mu = \KL(\tilde \nu \Vert \mu) \geq \KL(\xi \Vert \mu) - 2n e^{-1/(32 \eps^2)}.
	$$
	A combination of the above finally gives
	$$
	\int f d \xi - \log \int e^{f} d \mu - \KL(\xi \Vert \mu) \geq - \left (\eps + 2 e^{-1/(32)\eps^2} \right ) n - 2^6 \Lip(f) \sqrt{\frac{n \Comp(f)}{\eps}}.
	$$
	We choose $\eps = \left (4 \frac{\Lip(f)^2 \Comp(f)}{n} \right )^{1/3}$. Remark that we may legitimately assume that $$64 \Lip(f)^{2/3} \Comp(f)^{1/3} n^{2/3} \leq n \log 2,$$
	since, otherwise, the result of the corollary holds trivially by taking $\xi$ to be supported on the point where $f$ attains its maximum. This amounts to $\eps < 2^{-4}$. Since $2 e^{-1/(32)\eps^2} \leq \eps$ for all $\eps \in (0, 2^{-4})$, we get the desired bound.	
\end{proof}

\begin{proof} [Proof of Theorem \ref{thm:largedev}]
	We begin by following similar lines to the proof of \cite[Theorem 1.1]{CD14}. Define
	$$
	h(x) = \begin{cases}
	2x + 1 &  x \leq -1 \\
	- x^2 & -1 \leq x \leq 0 \\
	0 & x \geq 0
	\end{cases}
	$$
	so that $h$ is concave and $|h'| \leq 2$. By a Taylor approximation argument we clearly have for $x_0, x \in \RR$,
	\begin{equation}\label{eq:taylor}
	\left |h(x_0+ x) - h(x_0) - x h'(x_0) \right | \leq x^2.
	\end{equation}
	Next, define
	$$
	\psi(x) = K h((x-t)/\delta)
	$$
	for $K = \phi_p(t-\delta) / n$. Thus $|\psi'| \leq \frac{2K}{\delta}$ and $|\psi''| \leq \frac{2 K}{\delta^2}$. 
	Now set
	$$
	g(y) = n \psi(f(y)/n), ~~ \forall y \in \DC.
	$$
	Our first goal is to give an estimate for the complexity $\Comp(g)$ in terms of $\Comp(f)$. Define
	$$
	\mathcal{A}_f = \{\nabla f(y): y \in \DC\}, ~~ \mathcal{A}_g = \{\nabla g(y): y \in \DC\}
	$$
	For $y \in \DC$ define $S_i(y)$ to be the unique point all of whose coordinates except for the $i$th coordinate are equal to the corresponding ones of $y$ (and the $i$-th coordinate is the negative of its concurrent). Using a Taylor approximation to $\psi$ at $f(y)/n$, for all $y=(y_1,\dots,y_n) \in \DC$ we have the bound
	\begin{align*}
	\left | \partial_i g(y) - \psi'(f(y)/n) \partial_i f(y) \right |  ~& =  \left | - \tfrac{y_i}{2} \bigl( g(S_i(y))-g(y) \bigr ) - \psi'(f(y)/n) \partial_i f(y) \right | \\
	& = \left | -n \tfrac{y_i}{2} \bigl(\psi(f(S_i(y)) / n) - \psi(f(y) / n) \bigr) - \psi'(f(y)/n) \partial_i f(y) \right | \\
	& \stackrel{\eqref{eq:taylor}}{\leq} \frac{K}{2 \delta^2 n } \Lip(f)^2,
	\end{align*}
	or in other words,
	$$
	\left | \nabla g(y) - \psi'(f(y)/n) \nabla f(y) \right | \leq \frac{K}{\delta^2 \sqrt{n}} \Lip(f)^2.
	$$
	Consequently, since $|\psi'(f(y)/n)| \leq \frac{2K}{\delta}$ for all $y \in \DC$, we have for every vector $v \in \RR^n$,
	$$
	\sup_{u \in \mathcal{A}_g} \langle v,u \rangle \leq \max \left (0, \frac{2K}{\delta} \sup_{u \in \mathcal{A}_f} \langle v,u \rangle \right ) + \frac{K}{\delta^2 \sqrt{n}} \Lip(f)^2 |v|.
	$$
	Now, since for a standard Gaussian random vector $\Gamma$ one has $\EE|\Gamma| \leq \sqrt{n}$, we get that
	\begin{equation}\label{eq:mwgsmall}
	\Comp(g) \leq \frac{2K}{\delta} \Comp(f) + \frac{K}{\delta^2} \Lip(f)^2.
	\end{equation}
	By the same considerations, we also have that
	\begin{equation}\label{eq:lipgsmall}
	\Lip(g) \leq \sup_{x \in \RR} |\psi'(x)| \Lip(f) \leq \frac{2K}{\delta} \Lip(f).
	\end{equation}
	Since $f(x) \geq tn \Rightarrow g(x)=0$ we get that, for $Y \sim \mu_p$,
	\begin{equation}\label{eq:largedev1}
	\PP(f(Y) \geq tn) \leq \EE[\exp(g(Y))] = \int_{\DC} \exp(g(y) + \log I(\vec p,y)) d \mu
	\end{equation}
	where $\vec p = (-1+2p) (1,\dots,1)$ and $I(x,y) = \prod_{i \in [n]} (1 + x_i y_i)$, so that $\log I(\vec p, y) = \sum_i \log(1-y_i+2p y_i)$. We also easily have
	$$
	\Lip(\log I(\vec p, \cdot)) \leq |\log (p(1-p))|.
	$$
	Invoking Corollary \ref{cor:main}, we learn that for some product measure $\xi$,
	{\small
		\begin{equation}\label{eq:applycor}
		\log \int_{\DC} \exp(g(y) + \log I(\vec p,y)) d \mu
		\end{equation}
		$$
		\leq \int g d \xi + \int \log I(\vec p,y) d \xi(y) - \KL(\xi \Vert \mu) + 64 (\Lip(g) + |\log(p)(1-p)|)^{2/3} \Comp(g)^{1/3} n^{2/3}.
		$$
	}
	An easy calculation shows that $ \int \log I(\vec p,y) d \xi(y) - \KL(\xi \Vert \mu) = - \KL(\xi \Vert \mu_p)$. Combining this with \eqref{eq:mwgsmall} and \eqref{eq:lipgsmall} we have that
	\begin{equation}\label{key}
	\log \int_{\DC} \exp(g(y)) d \mu_p \leq \int g d \xi - \KL(\xi \Vert \mu_p) + 64 K L n^{2/3}
	\end{equation}
	with
	$$
	L := \frac{1}{\delta} \left ( 2 \Lip(f) + |\log(p(1-p))| \right )^{2/3} \left ( 2 \Comp(f) + \frac{1}{\delta} \Lip(f)^2 \right )^{1/3}
	$$
	where we use the assumption $\delta \leq \frac 1 n \phi_p(t-\delta)$ which implies $K/\delta \geq 1$. Next, we claim that for every product measure $\xi$ one has
	\begin{equation}\label{eq:gfp}
	\int g d \xi - \KL(\xi \Vert \mu_p) \leq - \phi_p(t-\delta).
	\end{equation}
	Indeed, if $\xi$ is such that $\int f d \xi \geq (t-\delta) n$ then by definition of $\phi_p$, we must have that $-\KL(\xi \Vert \mu_p) \leq -\phi_p(t-\delta)$, and since $g$ is non-positive, the inequality is correct. So, we may assume that $\int f d \xi \leq (t-\delta) n$. By the concavity and monotonicity of $\psi$ and by Jensen's inequality,
	we have
	$$
	\int g d \xi = \int n \psi(f(y)/n) d \xi(y) \leq  n \psi \left (\int f d \xi / n \right ) \leq n \psi \left ( t-\delta \right ) =-K n = -\phi_p(t-\delta).
	$$
	This establishes \eqref{eq:gfp}. Together with equations \eqref{eq:largedev1} and \eqref{eq:applycor}, we finally get
	$$
	\log \PP(f(Y) \geq tn) \leq - \phi_p(t-\delta) (1 - 64 L n^{-1/3}).
	$$
	The proof of the upper bound is complete. 

	Moving on to the lower bound, we define the function $g$ in the same way, with the exception that this time we take $K = 2 (\phi_p(t)+1)/n$. With the help of equation \eqref{eq:taylor} and by Jensen's inequality we have for every random variable $Z$,
	$$
	\EE \left [n \psi (Z / n) \right ] \geq n ( \psi(\EE[Z] / n)) - \frac{K}{\delta^2 n} \mathrm{Var}[Z].
	$$	
	Moreover, clearly if $W_1,\dots,W_n$ are independent Bernoulli random variables (with arbitrary expectation) and $W=(W_1,\dots,W_n)$ setting $M_i = \EE[f(W) | W_1,\dots,W_i]$ then
	\begin{align*}
	\mathrm{Var}(f(W)) = \sum_{i=1}^n \EE (M_{i}-M_{i-1})^2 \leq \Lip(f)^2 \sum_{i \in [n]} \mathrm{Var}[W_i] \leq n \Lip(f)^2.
	\end{align*}
	The two last displays teach us that for every product measure $\xi$ one has
	$$
	\int g d \xi \geq n \psi \left ( \frac{\int f d \xi} {n} \right ) - \frac{K}{\delta^2 } \Lip(f)^2.
	$$
	which finally gives
	\begin{align*}
	\PP(f(Y) \geq (t - \delta) n) ~& \geq \int e^g d \mu_p - \int_{\{y: f(y) \leq (t-\delta) n \} } e^{g} d \mu_p  \\
	& \geq \int e^g d \mu_p - \exp(-2 n K) \\
	& \geq \exp \left (\sup_{\xi \in \mathcal{PM}} \int g d \xi - \KL(\xi \Vert \mu_p)\right ) - \exp(- 2\phi_p(t) - 2) \\
	& \geq \exp \left (\sup_{\xi \in \mathcal{PM}} \left ( n \psi \left ( \frac{\int f d \xi} {n} \right ) - \KL(\xi \Vert \mu_p) \right ) - \frac{2 K}{\delta^2 } \Lip(f)^2 \right ) - \exp(-2\phi_p(t)-2) \\ 
	& \geq \exp \left (-\phi_p(t) \left ( 1 + \frac{1}{n \delta^2 } \Lip(f)^2 \right ) - \frac{1}{n \delta^2} \Lip(f)^2  \right ) - \exp(- 2\phi_p(t)-2) \\
	& \geq \exp \left (-\phi_p(t) \left ( 1 + \frac{1}{n \delta^2 } \Lip(f)^2 \right ) - 1 \right ) \Bigl (1 - \exp(-1) \Bigr )
	\end{align*}
	where the last inequality uses the assumption $\frac{2}{n \delta^2 } \Lip(f)^2 \leq 1$. We finally get
	$$
	\log \PP(f(Y) \geq (t - \delta) n) \geq -\phi_p(t) \left ( 1 + \frac{2}{n \delta^2 } \Lip(f)^2 \right ) - 2.
	$$
	The proof is complete.	
\end{proof}

\section{Exponential random graphs and complexity of subgraph counts}

Fix an integer $N$ and set $n = {N \choose 2}$. Recall that we define by $\mathcal{P}$ the set of above-diagonal sequences $(y_{i,j})_{1 \leq i < j \leq N}$, which is identified with $C_n$. For a finite, simple graph $H$ and a point $y \in \DC$ we define
$$
f_H(y) = N^2 t(H,G_y).
$$
where $G_y$ is the graph associated with $y$ and $t(H,G)$ is the homomorphism density defined in section \ref{sec:expgraphs}. The next lemma gives a bound for the Gaussian-width complexity of subgraph counts.
\begin{lemma} \label{lem:subgraphcomplexity}
	For every finite simple graph $H=([k],E)$ on $k$ vertices, we have $\Comp(f_H) \leq |E| N^{3/2}$.
\end{lemma}
\begin{proof}
	For a fixed $y \in \DC$, let $A=A(y)$ be the adjacency matrix of $G_y$. We have
	$$
	f_H(y) = \frac{1}{N^{k-2}} \sum_{q \in [N]^k} \prod_{(\ell,\ell') \in E} A_{q_\ell, q_{\ell'}}.
	$$
	A calculation gives (see also \cite[Equation (5.2)]{CD14})
	\begin{align} \label{eq:calcder}
	\frac{\partial f_H(y)}{\partial y_{i,j}} ~& = \frac{1}{N^{k-2}} \sum_{(a,b) \in E} \sum_{q \in [N]^{k} \atop q_a = i, q_b = j} \prod_{(\ell,\ell') \in E \atop \{\ell,\ell'\} \neq \{a,b\}} A_{q_\ell, q_{\ell'}} \\
	& = \frac{1}{N^{k-2}} \sum_{(a,b) \in E} \sum_{q \in [N]^{k-2}} \prod_{(a,\ell) \in E \atop \ell \neq b} A_{i, q_{\ell}} \prod_{(b,\ell) \in E \atop \ell \neq a} \nonumber A_{j, q_{\ell}} \prod_{(\ell, \ell') \in E \atop \{a,b\} \cap \{\ell, \ell' \} = \emptyset} A_{q_\ell, q_{\ell'}} \nonumber \\
	& = \frac{1}{N^{k-2}} \sum_{(a,b) \in E} \sum_{q \in [N]^{k-2}} J_{i,j}(A,a,b,q) \nonumber
	\end{align}
	where 
	$$
	J_{i,j}(A,a,b,q) = \prod_{(a,\ell) \in E \atop \ell \neq b} A_{i, q_{\ell}} \prod_{(b,\ell) \in E \atop \ell \neq a} A_{j, q_{\ell}} \prod_{(\ell, \ell') \in E \atop \{a,b\} \cap \{\ell, \ell' \} = \emptyset} A_{q_\ell, q_{\ell'}}, ~~ \forall (a,b) \in E, ~ q \in [N]^{k-2}.
	$$
	By construction, we see that for all $a,b,q$ the matrix $J(A,a,b,q)$ is of rank $1$. Moreover, clearly the entries of this matrix are all in $\{0,1\}$. Consequently, we have for all $M \in \mathbb{M}_{N \times N}$
	$$
	\sup_{A \in \mathbb{M}_{n \times n} \atop q \in [N]^{k-1}, (a,b) \in E} \tr(J_{i,j}(A,a,b,q) M) \leq \sup_{u,v \in \{0,1\}^n} \langle u, M v \rangle \leq N \Vert M \Vert_{\mathrm{OP}}.
	$$
	Thus if for $\theta \in \RR^n$ we set $M(\theta)$ to be the unique symmetric matrix with null diagonal and whose above-diagonal entries that those of $\theta$, then the above inequality combined with \eqref{eq:calcder} yields
	\begin{align*}
	\langle \nabla f_H(y), \theta \rangle ~& = \frac{1}{N^{k-2}} \sum_{(a,b) \in E} \sum_{q \in [N]^{k-2}} \tr(J_{i,j}(A,a,b,q) M(\theta)) \\
	& \leq \frac{1}{2 N^{k-2}} \sum_{(a,b) \in E} \sum_{q \in [N]^{k-2}} N \Vert M(\theta) \Vert_{\mathrm{OP}} \\
	& \leq \frac{1}{2}|E| \cdot N \Vert M(\theta) \Vert_{\mathrm{OP}}.
	\end{align*}
	Therefore, if $\Gamma$ is a standard Gaussian random vector in $\RR^n$ then
	$$
	\EE \sup_{y \in \DC} \langle \nabla f_H(y), \Gamma \rangle \leq  \frac{1}{2}|E| \cdot N \EE \Vert M(\Gamma) \Vert_{\mathrm{OP}} \leq |E| N^{3/2}.
	$$
	where the last inequality follows from standard estimates for norms of Gaussian matrices. The proof is complete.
\end{proof}

We can now prove our Theorem about decomposition of exponential random graphs.
\begin{proof}[Proof of Theorem \ref{thm:exponential}]
Setting $n = {N \choose 2}$, we identify a point $y \in \DC$ with the graph $G_y$ as in Section \ref{sec:triangles}. Consider the function
$$
f(y) = N^2 \sum_{i=1}^l \beta_i t(H_i,G_y).
$$
According to Lemma \ref{lem:subgraphcomplexity} and since the Gaussian-width is sub-additive, the complexity of $f$ is bounded as follows
$$
\Comp(f) \leq 2 n^{3/4} \sum_{i=1}^l |\beta_i| |E(H_i)|.
$$
Define the measure $\nu$ on $\DC$ by $d \nu = \frac{e^{f} d \mu}{\int_{\DC} e^f d \mu}$. Observe that by definition if $Y \sim \nu$ then $G_Y \stackrel{d}{=} G$.
We now apply Theorem \ref{thm:maindecomp} with the measure $\nu$, $\eps$ and some $\alpha > 1$ whose value will be chosen later on, to obtain a measure $m$ on $B(0, \eps \sqrt{n})$. Moreover, for all $\theta \in \RR^n$ define by $\xi_\theta$ the product measure having the same marginals as $\tilt_\theta \nu$, and define by $\vec p(\theta)$ to be the unique $\vec p \in [0,1]^{ {N \choose 2} }$ such that $G(N,\vec p)$ has the same distribution as $G_Z$ with $Z \sim \xi$. Let $\rho$ be push-forward of $m$ under the map $\theta \to \vec p(\theta)$. For all $\theta$ let $Y_\theta$ be a random point having the law $\tilt_\theta \nu$. Moreover let $\tilde \theta$ be a random variable in $\RR^n$ whose law is $m$, which is independent from the family $\{Y_\theta\}$. Then by equation \eqref{eq:decomp1} we have that
$$
G := G_{Y_{\tilde \theta}}
$$
has the exponential graph distribution \eqref{eq:defexpgraph}. Moreover, for all $\theta \in \RR^n$ let $G'_\theta$ be distributed as $G(N, \vec p(\theta))$. We may assume that $G'_\theta$ are defined on the same probability space, and that the family $\{G'_\theta\}$ is independent of $\tilde \theta$. Now, define
$$
G' = G'_{\tilde \theta}.
$$
It is clear that $G'$ has the distribution $G(N,\rho)$. Now, according to equation \eqref{eq:thmmd-2} there exists $\Theta$ with $m(\Theta) \geq 1 - \frac{1}{\alpha} - \frac{1}{n}$ such that for all $\theta \in \Theta$ we may couple the graph $G'_\theta$ with $G_{Y_\theta}$ so that
$$
\EE d_H\bigl (G_{Y_\theta}, G'_\theta \bigr ) \leq 16 \sqrt{\frac{\alpha n \Comp(\nu)}{\eps}}.
$$
Therefore, by taking expectation with respect to $\tilde \theta$, we have
$$
\EE d_H\bigl (G, G' \bigr ) \leq 16 \sqrt{\frac{\alpha n \Comp(\nu)}{\eps}} + n  \left (\frac{1}{\alpha} + \frac{1}{n} \right ).
$$
Choosing $\alpha = \frac{n^{1/3} \eps^{1/3}}{\Comp(\nu)^{1/3}}$, we get
$$
\EE d_H\bigl (G, G' \bigr ) \leq \frac{20 n^{2/3} \Comp(\nu)^{1/3}}{\eps^{1/3}} \leq \frac{20 n^{11/12} }{\eps^{1/3}} \left (\sum_{i=1}^l |\beta_i| |E(H_i)|\right )^{1/3}
$$
and finally equation \eqref{eq:thmmd-3} tells us that
$$
\int I(p) d \rho(p) \geq \Ent(G) - 2 \eps n.
$$

\end{proof}

\section{Appendix - proofs of Auxiliary results}
\begin{proof}[Proof of Fact \ref{fact:chain}]
	This is a straightforward consequence of the chain rule for relative entrory. Let $Y=(Y_1,\dots,Y_n) \sim \nu$ and $\tilde Y= (\tilde Y_1,\dots,\tilde Y_n) \sim \xi$. Then we have
	\begin{align*}
	n \log 2 - \KL(\nu \Vert \mu) ~& = H((Y_1,\dots,Y_n)) \\
	& = \sum_{i=1}^n H(Y_i | (Y_1,...,Y_{i-1})) \\
	& \leq \sum_{i=1}^n H(Y_i) \\
	& = \sum_{i=1}^n H(\tilde Y_i) = n \log 2 - \KL(\xi \Vert \mu).
	\end{align*}
\end{proof}

\begin{proof}[Proof of Lemma \ref{lem:SF}]
	By definition of the function $g(y) = g_\nu(y)$ (or according to equation \eqref{eq:defgp}), we have that for all $y \in \DC$ if we denote $g_\nu(y)=(g_1,...,g_n)$ and $\nabla f_\nu(y) = (d_1,...,d_n)$ then
	$$
	g_i = \varphi(d_i) := \frac{e^{d_i} - 1}{e^{d_i} + 1 }.
	$$
	Since $|\varphi'(x)| \leq 1$ for all $x \in \RR$, we have that the set $\left \{ g_\nu(y) : y \in \DC  \right \}$ is the image of the set $\left \{ \nabla f(y) : y \in \DC  \right \}$ under a $1$-Lipschitz mapping. In turn, we have for all $y_1,y_2 \in \DC$,
	$$
	|g_\nu(y_1) - g_\nu(y_2)| \leq |\nabla f_\nu(y_1) - \nabla f_\nu(y_2)|
	$$
	and therefore	
	$$
	\EE \left [ \langle g(y_1) - g(y_2), \Gamma \rangle^2 \right ] \leq \EE \left [ \langle \nabla f (y_1) - \nabla f(y_2), \Gamma \rangle^2 \right ]
	$$
	where $\Gamma$ is a standard Gaussian random vector. An application of the Sudakov-Fernique inequality (see e.g., \cite[Chapter 3]{LedouxTalagBook}) completes the proof.
\end{proof}

\begin{proof}[Proof of Fact \ref{fact:dgt}]
	The proof is a straightforward calculation using It\^{o}'s formula,
	\begin{align}
	d g_t ~& = \frac{d q(X_t)}{h(X_t)} - \frac{q(X_t) d h(X_t) }{ h^2(X_t) } + \frac{q(X_t) d[h(X)]_t }{h^3(X_t)} - \frac{d[q(X),h(X)]_t}{h^2(X_t)} \nonumber \\
	& = \left  ( \frac{\nabla q(X_t) }{ h(X_t)} - \frac{q(X_t) \otimes \nabla h(X_t) }{h^2(X_t)} \right ) \left ( \sigma_t^{1/2} d B_t + \sigma_t v_t dt \right ) \nonumber \\
	& ~~~~ + \frac{q(X_t) \langle \sigma_t \nabla h(X_t), \nabla h(X_t) \rangle }{h^3(X_t)} dt - \frac{ \nabla q(X_t) \sigma_t \nabla h(X_t)  }{h^2(X_t)} dt \nonumber \\
	& = \left  ( \frac{\nabla q(X_t) }{ h(X_t)} - g_t \otimes v_t \right ) \left ( \sigma_t^{1/2} d B_t + \sigma_t v_t dt \right ) \nonumber \\
	& ~~~~ + g_t \langle v_t, \sigma_t v_t \rangle dt - \frac{ \nabla q(X_t) \sigma_t v_t  }{h(X_t)} dt \nonumber \\
	& = \left  ( \frac{\nabla q(X_t) }{ h(X_t)} - g_t \otimes v_t \right ) \sigma_t^{1/2} d B_t. \nonumber
	\end{align}
\end{proof}

\begin{proof}[Proof of Lemma \ref{lemHtHt}]
	Let $Y \in \DC$ be a random variable with law $\tilde \nu$. Fix $i \in [n]$. The lemma will be concluded by showing that 
	\begin{equation}\label{eq:needtoshowhtht}
	\mathrm{Var}  \langle g_{\tilt_\theta \nu}(Y), e_i \rangle \leq \exp(4 \| \theta \|_\infty ) \mathrm{Var}  \langle g_{\nu}(Y), e_i \rangle.
	\end{equation}
	Next, by definition of $\tilt_\theta \nu$ we have $\nabla f_{\tilt_\theta \nu} = \nabla f_{\nu} + \theta$, and equation \eqref{eq:tanh} gives
	$$
	g_{\tilt_\theta \nu}(y) = \tanh( \nabla f_{\tilt_\theta \nu}(y) ) = \tanh( \nabla f_{\nu}(y) + \theta ) = \tanh( \tanh^{-1}(g_{\nu}(y) ) + \theta )
	$$
	for all $y \in \DC$, so that 
	$$
	\langle g_{\tilt_\theta \nu}(y), e_i \rangle = u_{\theta_i}( \langle g_{\nu}(y), e_i \rangle ), ~~ \forall y \in \DC,$$ 
	where $u_s(z) := \tanh(\tanh^{-1}(z) + s)$.
	
	Defining $Z = \langle g_{\nu}(Y), e_i \rangle$ and using the last display, equation \eqref{eq:needtoshowhtht} becomes
	$$
	\mathrm{Var} \left [ u_{\theta_i}(Z) \right ] \leq \exp(4 \| \theta \|_\infty ) \mathrm{Var} [Z].
	$$
	It is straightforward to check that for all $x \in (-1,1)$ and $s \in \RR$, one has $\left | \frac{d}{dx} u_s(x) \right | \leq \exp^{2 |s|}$. Since for any $L$-Lipschitz function $u$ and any random variable $Z$ one has $\rm{Var} [u(Z)] \leq L^2 \mathrm{Var} [Z]$, it follows that
	$$
	\mathrm{Var} \left [ u_{\theta_i}(Z) \right ] \leq \exp(4 |\theta_i|) \mathrm{Var} [Z].
	$$
	This yields \eqref{eq:needtoshowhtht} which completes the proof.
\end{proof}

\begin{proof} [Proof of Lemma \ref{lem:vtcoord}]
	Define $x = \langle X_t, e_i \rangle$, $V = \langle v_\infty, e_i \rangle$ and $G = \langle g_\infty, e_i \rangle$. Let $\tilde p: \DC \to \RR$ be a function satisfying $\partial_i \tilde p \equiv 0$, let $p$ be the harmonic extension of $\tilde p$ to $\CC$ and define $P = \tilde p(X_\infty)$. Set $\tilde \DC = \{ (y_1,...,y_n), ~ y_i = \langle X_t, e_i \rangle \mbox{ and } y_j \in \{-1,1\}, \forall j \neq i \}$ and let $\pi:\{-1,1\}^n \to \{-1,1\}^{n-1}$ be the projection defined by
	$\pi((y_1,...,y_n)) = (y_1,...,y_{i-1},y_{i+1},...,y_n)$. We calculate
	\begin{align*}
	\EE[V P | \FF_t ] ~& \stackrel{ \eqref{eq:com}}{=} \frac{1}{h(X_t)} \sum_{y \in \DC} w(X_t, y) e^{f(y)} p(y) \langle v(y), e_i \rangle \\
	& = \frac{1}{h(X_t)} \sum_{y \in \DC} w(X_t, y) p(y) \partial_i e^f(y) \\
	& = \frac{1}{h(X_t)} \sum_{\tilde y \in \tilde \DC} w(\pi(X_t), \pi(\tilde y)) p(\tilde y) \partial_i h(\tilde y) \\
	& = \frac{1}{h(X_t)} \sum_{\tilde y \in \tilde \DC}  w(\pi(X_t), \pi(\tilde y)) h(\tilde y) p(\tilde y) \frac{\partial_i h(\tilde y)}{h(\tilde y)} \\
	& = \frac{1}{h(X_t)} \sum_{y \in \DC} w(X_t, y) h(y) p(y) \frac{\partial_i h(y)}{h \bigl( (y_1,\dots,y_{i-1},x,y_{i+1},\dots,y_n) \bigr )} \\
	& \stackrel{ \eqref{eq:gv} }{ = } \frac{1}{h(X_t)} \sum_{y \in \DC} w(X_t, y) h(y) p(y) \zeta_{x}(\langle g(y), e_i \rangle) \stackrel{ \eqref{eq:com}}{=} \EE[ \zeta_{x}(G) P | \FF_t].
	\end{align*}
	The proof of \eqref{eq:vtcoord} follows by taking $\tilde p(y) = 1$ and that of \eqref{eq:vtcoord2} follows by taking $\tilde p(y) = \langle g(y) - g_t, e_i \rangle$. 
	Finally, in order to obtain the bound \eqref{eq:vtbound}, we combine the formula \eqref{eq:vtcoord} with the estimate
	$$
	|\zeta_x(g)| = \left | \frac{g}{1+gx} \right | \leq \frac{1}{1-\tfrac{1}{2}} \leq 2, ~~ \forall g \in [-1,1], x \in [-1/2,1/2].
	$$
	To obtain the bound \eqref{eq:vtbound2}, observe that for any $x,g \in (0,1)$ one has	
	$$
	x \cdot \zeta_x(g) = \frac{gx}{1+gx} \leq 1,
	$$
	use the formula \eqref{eq:vtcoord} and take expectation.

\end{proof}

\bibliographystyle{plainnat}
\bibliography{bib}

\end{document}